\documentclass[12pt,a4paper,twoside,reqno]{amsart} 
 \usepackage{amsfonts,amssymb,amscd,amsmath,enumerate,verbatim,calc} 

\usepackage{color,psfrag}

\newtheorem{proposition}{Proposition}[section]
\newtheorem{theorem}[proposition]{Theorem}
\newtheorem{corollary}[proposition]{Corollary}
\newtheorem{definition}[proposition]{Definition}
\newtheorem{lemma}[proposition]{Lemma}

\allowdisplaybreaks

\theoremstyle{remark}
\newtheorem{remark}{Remark}

\begin{document}
\title{On a Tur{\'a}n's theorem for small primes}
\author{Tokuhon Makoto Minamide, Haruka Sakai,\\
 and\\
 Yoshio Tanigawa}
\thanks{2020 Mathematics Subject Classification: 11N37, 11N36\\
Key words and phrases: The number of distinct prime divisors, Tur{\'a}n's theorem, The method of Granville and Soundararajan\\
This work is supported by JSPS KAKENHI No. 22K03245.}
\maketitle
\begin{abstract}
Denote by $\omega(n)$ the number of distinct prime divisors of the natural number $n$.
In 2007, Granville and Soundararajan gave a quite new method to compute the higher  moments
$\sum_{n\leq x}(\omega(n)-\log\log x)^{k}$, for a wide range of integers $k\geq 2$. 
In this notes, we shall apply the method for $\omega_{z}(n)$ which denotes the number of distinct prime divisors $\leq z$ of $n$.
Especially, for odd integers $k\geq 3$, we lead asymptotic formulas for $\sum_{n\leq x}(\omega_{z}(n)-\log\log z)^{k}$,
under certain restrictions on $k$, $z$, and $x$. Also, we refer to a framework of the approach.  
\end{abstract}
%
%
%
%
%
\section{Introduction}
Let $\omega(n)$ denote the number of distinct prime divisors of a natural number $n$. 
Hardy and Ramanujan studied the behaviour of $\omega(n)$ in the literature \cite{HR} and proved a theorem:

\textit{
For any given $\varepsilon$ ($0<\varepsilon <1$), the number of $n\leq x$ which does not satisfy
\begin{align*}
(1-\varepsilon)\log\log n <\omega (n) <(1+\varepsilon)\log\log n
\end{align*}
is $o(x)$, as $x\to\infty$.
}

Regarding this theorem Tur{\'a}n gave a simple proof in \cite{turan} 
based on the following formula.
\begin{align}\label{turan-th}
\sum_{n\leq x}(\omega(n) -\log\log x)^{2}=O(x\log\log x), \quad (x\to\infty).
\end{align}
See, \cite[p.~474, Theorem 431]{HW}, also. Tur{\'a}n's method  
has been developed by many authors and various results  
were obtained, for instance,
by Erd{\H o}s \cite{E}, Er{\H o}ds and Kac \cite{EK}, Liu \cite{YRLiu},
Murty and Murty \cite{MuMu}, Murty and Saidak \cite{MuSa}.
Concerning the Erd{\H o}s-Kac theorem \cite{EK},  
as a generalization of (\ref{turan-th}), the higher power moments 
\begin{align}\label{turan-th-hm}
\sum_{n\leq x}(\omega(n)-\log\log x)^{k}\quad (x\to\infty, k=2,3,\ldots)
\end{align}
were investigated by Delange \cite{D}, \cite{D2} and Halberstam \cite{Hal}. 
It is natural to ask a formula in (\ref{turan-th-hm}) 
for wider range of $k$.  
Indeed, Granville and Soundararajan obtained such results uniformly in $k$ and $x$ under a certain restriction. 
Their method is quit new.
The authors mentioned that it is inspired by \cite{MoSo} due to Montgomery and Soundararajan. See Theorem \ref{GS-th-1} below.
To state the theorem we shall put for any integer $j\geq 1$
\begin{align}\label{teigi-C-j}
C_{j}:=\frac{\Gamma(j+1)}{2^{\frac{j}{2}}\Gamma \left(\frac{j}{2}+1\right)}=\frac{2^{\frac{j}{2}}}{\sqrt{\pi}}\Gamma \left(\frac{j}{2}+\frac{1}{2}\right),
\end{align}
where $\Gamma (\cdot)$ denotes the Gamma function, 
the right-hand side of (\ref{teigi-C-j})
is derived from the duplication formula \cite[p.~57]{Tit}. For convenience, we write $C_{0}:=1$.
It is easily seen that
\begin{align*}
C_{1}=\sqrt{\frac{2}{\pi}},\quad C_{2}=1,\quad C_{3}=2\sqrt{\frac{2}{\pi}},
\end{align*}
\begin{align*}
C_{j}=\sqrt{2}\left(\frac{j}{e}\right)^{\frac{j}{2}}\frac{\left(1+\frac{1}{j}\right)^{\frac{j}{2}}}{\sqrt{e}}(1+o(1)) \sim \sqrt{2}\left(\frac{j}{e}\right)^\frac{j}{2} \gg 2^{j} \quad (j\to \infty),
\end{align*}
and
\begin{align*}
\frac{C_{j-1}}{C_{j}}=\frac{1}{\sqrt{2}}\frac{\Gamma \left(\frac{j}{2}\right)}{\Gamma \left(\frac{j}{2}+\frac{1}{2}\right)} \asymp \frac{1}{\sqrt{j}} \quad (j\to\infty),
\end{align*}
by the Stirling formula and Lemma in \cite[p.~58]{Tit}, respectively.
As for sequences $\{a_{j}\}$, $\{b_{j}\}$, generally, 
$a_{j}\sim b_{j}$ denotes $a_{j}/b_{j}\to 1$ as $j\to\infty$, and $a_{j} \asymp b_{j}$ denotes $a_{j}\ll b_{j}$ and $a_{j}\gg b_{j}$, as $j\to\infty$.
These estimates on $C_{j}$ are used in this notes.

Concerning the higer moments (\ref{turan-th-hm}), 
Granville and Soundararajan showed the following theorem.
\begin{theorem}[{Granville and Soundararajan \cite[p.~17, Theorem 1]{GS}}]\label{GS-th-1}
Keep the notation as above.
Let $x\geq 1$ be sufficiently large. For positive integers $k\geq 2$, we assume that $k\leq (\log\log (x^{1/k}))^{1/3}$.
Then, we have the following (a) and (b), uniformly in $k$ and $x$. 
\begin{enumerate}
\item[\rm (a)]For even integers $k\geq 2$, we have
\begin{align*}
\sum_{n\leq x}\left(\omega(n) -\log\log x\right)^{k}
=C_{k} x (\log\log x)^{\frac{k}{2}} \left(1+O\left(\frac{k^{3/2}}{(\log\log x)^{1/2}}\right)\right).
\end{align*}  
\item[\rm (b)] For odd integers $k\geq 3$, we have
\begin{align*}
\sum_{n\leq x}(\omega (n) -\log\log x)^{k} \ll C_{k}x(\log \log x)^{\frac{k}{2}}\frac{k^{3/2}}{(\log\log x)^{1/2}}.
\end{align*}
\end{enumerate}
\end{theorem}

In this notes, applying the method of the proof of Theorem \ref{GS-th-1} in \cite{GS} we shall obtain
a certain Tur{\'a}n's theorem for small primes. Define
\begin{align}\label{teigi-omega-z}
\omega_{z}(n):=\sum_{\begin{subarray}{c}p|n \\ p\leq z\end{subarray}}1
\end{align}
for any natural number $n$ and any real number $z \geq 1$. 
Instead of (\ref{turan-th-hm}), we shall investigate the moments
\begin{align*}
\sum_{n\leq x}(\omega_{z}(n)-\log\log z)^{k},
\end{align*}
uniformly in $k$, $z$, and $x$ under certain conditions.
To state our theorem define 
\begin{align}\label{teigi-teisu-B}
B:=\lim_{x\to\infty}\left(\sum_{p\leq x}\frac{1}{p}-\log\log x\right)\ (\approx 0.261),
\end{align}
and for each integer $j\geq 3$ 
\begin{align}\label{teigi-C-j-tilde}
\tilde{C}_{j}:=\frac{\sqrt{2}}{6}\frac{(j-1)\Gamma\left(\frac{j}{2}+1\right)}{\Gamma \left(\frac{j}{2}+\frac{1}{2}\right)}C_{j}
=\frac{2^{\frac{j-1}{2}}(j-1)}{3\sqrt{\pi}}\Gamma \left(\frac{j}{2}+1\right),
\end{align}
here, note that
\begin{align*}
\tilde{C}_{j} \asymp j^{\frac{3}{2}}C_{j}\quad \textit{and}\quad \tilde{C}_{3}=1.
\end{align*}

We shall prove the following theorem.
\begin{theorem}\label{Hatsumi-1}
Keep the notation as above.
Let $x\geq 1$ and $z\geq 1$ be sufficiently large. For positive integers $k\geq 2$, we assume that
$k \leq (\log \log z)^{1/3} $ and  $e^{e^{k^3}}\leq z\leq x^{1/k}$. Then, we have
the following (a) and (b), uniformly in $k$, $z$, and $x$.
\begin{enumerate}
\item[\rm (a)] For even integers $k\geq 2$, we have
\begin{align*}
\sum_{n\leq x}(\omega_{z}(n)-\log\log z)^{k} = C_{k} x(\log \log z)^{\frac{k}{2}} \left( 1+O\left( \frac{k^{3}}{\log\log z}\right) \right).
\end{align*}
\item[\rm (b)] For odd integers $k\geq 3$, we have three formulas (I), (II), and (III) as follows.\\
 
\begin{enumerate}
\item[\rm (I)]If $3\leq k \leq (\log\log z)^{1/7}$, then we have
\begin{align*}
&\sum_{n\leq x}(\omega_{z}(n)-\log\log z)^{k}\\
&=\tilde{C}_{k}x (\log\log z) ^{\frac{k-1}{2}} +B(k^{3/2}C_{k}) x (\log\log z)^{\frac{k-1}{2}} \left(\frac{C_{k-1}}{k^{1/2} C_{k}}\right)\\
&\quad +O\left((k^{3/2}C_{k}) x (\log\log z)^{\frac{k-1}{2}} \cdot \frac{1}{(k\log\log z)^{1/2}}\right).
\end{align*}
\item[\rm (II)] If $(\log\log z)^{1/7}<k <(\log\log z)^{1/4}$, then we have
\begin{align*}
&\sum_{n\leq x}(\omega_{z}(n)-\log\log z)^{k} \\
&=\tilde{C}_{k} x (\log\log z)^{\frac{k-1}{2}} + B \left(k^{3/2}C_{k}\right) x (\log\log z)^{\frac{k-1}{2}} \left(\frac{C_{k-1}}{k^{1/2}C_{k}}\right)\\
&\quad +O\left( \left(k^{3/2}C_{k}\right) x (\log\log z)^{\frac{k-1}{2}}\frac{k^{3}}{\log\log z}\right).
\end{align*}
\item[\rm (III)] If $(\log\log z)^{1/4}\leq k \leq (\log\log z)^{1/3}$, then we have
\begin{align*}
&\sum_{n\leq x}\left(\omega_{z}(n)-\log\log z\right)^{k}\\
&= \tilde{C}_{k} x (\log\log z)^{\frac{k-1}{2}} +O\left(\left(k^{3/2}C_{k}\right)x (\log\log z)^{\frac{k-1}{2}} \frac{k^{3}}{\log\log z}\right).
\end{align*}
\end{enumerate}
\end{enumerate}
\end{theorem}

\bigskip

\begin{remark}
According to Theorem \ref{GS-th-1}, $k$th power moment of $\omega(n)-\log\log x$ has a main term when $k$ is even,
while not when $k$ is odd. However, if we pose a restriction on prime divisors of $n$ as in (\ref{teigi-omega-z}),
$k$th power moments of $\omega_{z}(n)-\log\log z$ has a main term in both cases that $k$ is even and odd. Moreover
for small $k$ we can find more accurate asymptotic formulas, see Theorem \ref{Sakai-Haruka-Theorem}, below.
\end{remark}

In Section 2, we reconsider Proposition 2 in \cite[p.~17]{GS} to show Theorem \ref{Hatsumi-1}. And we prove Theorem \ref{Hatsumi-1} in Section 3. Moreover,
in Section 4 we shall also discuss a framework of the theorem.

\section{On Proposition 2 by Granville and Soundararajan}
To deduce (b) of Theorem \ref{Hatsumi-1} which is the main result of this notes we shall reconsider Proposition 2 given by Granville and Soundararajan in \cite[p.~17]{GS}.
To state their proposition we fix some notation. Let $\pi(z)$ denote the number of primes up to $z$, as usual. 
Moreover, we recall the function $f_{p}(n)$
which is introduced in \cite{GS}:
 
For any natural number $n$ and any prime $p$, the $f_{p}(n)$ is defined as
\begin{align}\label{teigi-f-1}
f_{p}(n):=
\begin{cases}
1-\frac{1}{p} & (\textit{if $p|n$}),\\
-\frac{1}{p}  & (\textit{if $p\nmid n$}).
\end{cases}
\end{align}
For convenience, we write
\begin{align}\label{teigi-f-2}
f_{p_{1}\cdots p_{l}}(n):=f_{p_{1}}(n)\cdots f_{p_{l}}(n)
\end{align}
for any primes $p_{1}, \ldots, p_{l}$ and any natural number $n$.

In \cite{GS}, Granville and Soundararajan showed the following proposition to deduce Theorem \ref{GS-th-1} (\cite[p.~17, Theorem 1]{GS}).
\begin{proposition}[{\cite[p.~17, Proposition 2]{GS}}]\label{GS-prop-2}
Let $k\geq 2$ be any integer, $x\geq 1$ and $z\geq 1$ real numbers.
Further assume that $k\leq (\log\log z)^{1/3}$.
Then, uniformly in $k$, $z$, and (x) we have the following (a) and (b).
\begin{enumerate}
\item[\rm (a)] For even integers $k\geq 2$,  we have
\begin{align*}
\sum_{n\leq x}\left(\sum_{p\leq z}f_{p}(n)\right)^{k}
=C_{k}x (\log\log z)^{\frac{k}{2}}\left(1+O\left(\frac{k^{3}}{\log\log z}\right)\right)+O\left(2^{k}\pi(z)^{k}\right).
\end{align*}
\item[\rm (b)]For odd integers $k\geq 3$, we have
\begin{align*}
\sum_{n\leq x} \left(\sum_{p\leq z}f_{p}(n)\right)^{k} \ll C_{k} x (\log\log z)^{\frac{k}{2}}\frac{k^{3/2}}{(\log\log z)^{1/2}} +2^{k}\pi(z)^{k}.
\end{align*}
\end{enumerate}
\end{proposition}

Concerning (b) of Proposition \ref{GS-prop-2}, we shall show the following lemma which is used in the proof of Theorem \ref{Hatsumi-1}.
\begin{lemma}[{cf. (b) of Proposition \ref{GS-prop-2}}]\label{Natsuki-1}
Let $k\geq 3$ be any odd integer, and $x$, $z\geq 1$ be sufficiently large.
Under the restriction $k\leq (\log\log z)^{1/3}$, we have, uniformly in $k$, $z$, and $x$ 
\begin{align*}
\sum_{n\leq x}\left(\sum_{p\leq z}f_{p}(n)\right)^{k}=\tilde{C}_{k}x(\log\log z)^{\frac{k-1}{2}} \left(1+O\left(\frac{k^{3}}{\log\log z}\right)\right) 
                                                      +O\left(2^{k}\pi(z)^{k}\right),
\end{align*}
where $\tilde{C}_{k}$ is the constant defined in (\ref{teigi-C-j-tilde}).
\end{lemma}
To prove this lemma we shall recall the function $G(r)$ which is introduced by \cite[p.~17]{GS} and some properties of $G(r)$ (Lemma \ref{G(r)-lemma}).
For any positive integer $r\geq 2$, we shall express $r=p_{1}\cdots p_{l}$ the product of all prime divisors of $r$ (so, $l$ is the total number of prime factors of $r$), 
and we put $R:=\prod_{p|r}p$, the product of the distinct prime divisors of $r$, the so-called  kernel of $r$. Define $G(r)$ by
\begin{align}\label{916-def}
G(r):=\frac{1}{R}\sum_{d|R}f_{p_{1}\cdots p_{l}}(d)\varphi\left(\frac{R}{d}\right),
\end{align} 
where $f_{p_{1}\cdots p_{l}}(\cdot)$ is the function defined by (\ref{teigi-f-1}) and (\ref{teigi-f-2}), and $\varphi(\cdot)$ is the Euler function, as usual.
We shall collect some properties on $G(r)$ from \cite{GS}.

\begin{lemma}[{\cite[p.~18, 19]{GS}}]\label{G(r)-lemma}
Let $r$ and $r^{\prime}$ be positive integers $\geq 2$. Then we have the following properties of $G(r)$.
\begin{enumerate}
\item[\rm (i)] If $(r,r^{\prime})=1$, then $G\left(rr^{\prime}\right)=G(r)G\left(r^{\prime}\right)$.
\item[\rm (ii)] For the prime factorization of $r$, if $r=\prod_{q^{\alpha}|| r}q^{\alpha}$, then
\begin{align*}
G(r)=\prod_{q^{\alpha} ||r} \left(\frac{1}{q}\left(1-\frac{1}{q}\right)^{\alpha}+\left(-\frac{1}{q}\right)^{\alpha} \left(1-\frac{1}{q}\right)\right).
\end{align*}
\item[\rm (iii)] If $r$ is not square-full (that is, the exponent of some prime is one), then $G(r)=0$.
\item[\rm (iv)] For any prime $q$ and any integer $\alpha \geq 2$, we have
\begin{align*}
0\leq G\left(q^{\alpha}\right) \leq \frac{1}{q}\left(1-\frac{1}{q}\right).
\end{align*}
\end{enumerate}
\end{lemma}
\begin{proof}
Since these properties are shown easily, the proofs of them are omitted in \cite{GS}.
However, here we shall prove (iv). In fact, it is important in the proof of Lemma \ref{Natsuki-1}.

By (ii), for any prime $q$ and any integer $\alpha \geq 1$ we have
\begin{align*}
G\left(q^{\alpha}\right)=\frac{1}{q}\left(1-\frac{1}{q}\right)^{\alpha} +\left(-\frac{1}{q}\right)^{\alpha} \left(1-\frac{1}{q}\right).
\end{align*} 
Here, if $\alpha \geq 2$ is even, then $G\left(q^{\alpha}\right)\geq 0$ is trivial and 
\begin{align*}
G\left(q^{\alpha}\right)=\frac{1}{q}\left(1-\frac{1}{q}\right)
                          \left(\left(1-\frac{1}{q}\right)^{\alpha -1}+\left(\frac{1}{q}\right)^{\alpha -1}\right) \leq \frac{1}{q}\left(1-\frac{1}{q}\right).
\end{align*}
On the other hand, if $\alpha\geq 3$ is odd, then trivially
\begin{align*}
G\left(q^{\alpha} \right)= \frac{1}{q}\left(1-\frac{1}{q}\right)^{\alpha} -\left(\frac{1}{q}\right)^{\alpha}\left(1-\frac{1}{q}\right) \leq \frac{1}{q}\left(1-\frac{1}{q}\right),
\end{align*}
and by noting that
\begin{align}\label{zenshin-H4}
1-\frac{1}{q} \geq \frac{1}{2} \geq \frac{1}{q},
\end{align}
we find the positivity
\begin{align*}
G\left(q^{\alpha}\right) =\frac{1}{q}\left(1-\frac{1}{q}\right) \left(\left(1-\frac{1}{q}\right)^{\alpha -1} -\left(\frac{1}{q}\right)^{\alpha -1}\right)\geq 0.
\end{align*}
\end{proof}

\begin{remark}\label{RRRemark} By (iv) of Lemma \ref{G(r)-lemma}, we obtain the positivity $G\left(q^{3}\right)\geq 0$, 
particularly, which is an important point in the proof of Lemma \ref{Natsuki-1}, and is required for the proof of (b) of Theorem \ref{Hatsumi-1}. 
In the proof of (a) of Theorem \ref{Hatsumi-1} ($k$: even), the positivity of $G\left(q^{3}\right)\geq 0$ is not asked. See Section 3.
In the last section of this notes (Section 4), we will discuss a framework of (b) of Theorem \ref{Hatsumi-1}. In which, we need an assumption corresponding to (\ref{zenshin-H4}).
See the hypothesis (H4) and (v) of Lemma \ref{lemma-III}, in Section 4.
\end{remark}

Following the arguments by Granville and Soundararajan in \cite[p.~18--20]{GS} we shall now prove Lemma \ref{Natsuki-1}.

%
%
%
\begin{proof}[Proof of Lemma \ref{Natsuki-1}]
First of all, let $k\geq 1$ be any integer. Using the notation (\ref{teigi-f-2}), we observe that
\begin{align}\label{hoshi-3}
\sum_{n\leq x}\left(\sum_{p\leq z}f_{p}(n)\right)^{k} =\sum_{p_{1},\ldots, p_{k}\leq z}\sum_{n\leq x}f_{p_{1}\cdots p_{k}}(n).
\end{align}
For each $k$-tuple $\{p_{1},\ldots, p_{k}\}$ of primes in $\sum_{p_{1},\ldots, p_{k}\leq z}$, we put $r=p_{1}\cdots p_{k}$ and 
write $r=\prod_{i=1}^{s}q_{i}^{\alpha_{i}}$ which is the prime factorization of $r$ (so, $s=\omega(r)$, $q_{1},\ldots, q_{s}$ are distinct primes, $\alpha_{i}$'s are
positive integers), further we put $R:=\prod_{i=1}^{s}q_{i}$ (the kernel of $r$). We note that if $d=(n,R)$ is the greatest common divisors of $n$ and $R$, then
\begin{align*}
f_{p_{1}\cdots p_{k}}(n)=f_{p_{1}\cdots p_{k}}(d).
\end{align*}  
Using them we see that
\begin{align}
(\textit{RHS of } (\ref{hoshi-3})) &= \sum_{p_{1},\ldots, p_{k}\leq z}\sum_{d|R}\sum_{\begin{subarray}{c}n\leq x\\ (n,R)=d\end{subarray}} f_{p_{1}\cdots p_{k}}(n)\nonumber \\
                                   &=\sum_{p_{1},\ldots, p_{k}\leq z}\sum_{d|R}f_{p_{1}\cdots p_{k}}(d)\sum_{\begin{subarray}{c}n\leq x\\ (n,R)=d\end{subarray}}1. 
\label{uru-sai-tonari}
\end{align}
Here, using the M{\"o}bius function $\mu(\cdot)$ we have
\begin{align*}
&\sum_{\begin{subarray}{c}n\leq x\\ (n,R)=d\end{subarray}}1=\sum_{\delta |\frac{R}{d}}\mu(\delta) \left[\frac{x}{\delta d}\right]
                                                             =\sum_{\delta |\frac{R}{d}}\mu(\delta) \left(\frac{x}{\delta d} +O(1)\right)\\
&=\frac{x}{d}\sum_{\delta |\frac{R}{d}}\frac{\mu(\delta)}{\delta} \frac{R/d}{R/d} +O\left(\sum_{\delta |\frac{R}{d}} |\mu(\delta)|\right) 
=\frac{x}{R} \varphi\left(\frac{R}{d}\right) +O\left(\sum_{\delta |\frac{R}{d}} |\mu(\delta)| \right), 
\end{align*}
and substitute this on the right-hand side of (\ref{uru-sai-tonari}), then
\begin{align*}
&\sum_{n\leq x}\left(\sum_{p\leq z}f_{p}(n)\right)^{k}\\
& =x\sum_{p_{1},\ldots, p_{k}\leq z}G(r)
+O\left(\sum_{p_{1},\ldots, p_{k}\leq z} \left(\sum_{d|R}|f_{p_{1}\cdots p_{k}}(d)|\sum_{\delta |\frac{R}{d}}|\mu(\delta)|\right)\right),
\end{align*}
here, $G(r)$ is the function defined by (\ref{916-def}). The above $O$-term is evaluated as $O\left(2^{k}\pi(z)^{k}\right)$. 
In fact, putting $l=\omega (d)$ in the $O$-term we observe that
\begin{align*}
&\sum_{d|R}|f_{p_{1}\cdots p_{k}}(d)|\sum_{\delta|\frac{R}{d}}|\mu(\delta)|\\
&\leq \sum_{d|R}\prod_{\begin{subarray}{c}p_{i}\\ p_{i}\nmid d\end{subarray}} \frac{1}{p_{i}}\cdot 2^{\omega \left(\frac{R}{d}\right)}
\leq \sum_{d|R}\left(\frac{1}{2}\right)^{\omega (R)-l}2^{\omega(R)-l} 
\leq 2^{k},
\end{align*}
then
\begin{align*}
\sum_{p_{1},\ldots, p_{k}\leq z} \left(\sum_{d|R}|f_{p_{1}\cdots p_{k}}(d)|\sum_{\delta |\frac{R}{d}}|\mu(\delta)|\right)
\leq \sum_{p_{1},\ldots, p_{k}\leq z}2^{k} \leq 2^{k}\pi(z)^{k}.
\end{align*}
On the other hand, we shall use (iii) of Lemma \ref{G(r)-lemma} to $\sum_{p_{1},\ldots, p_{k}\leq z}G(r)$. 
Finally, we get
\begin{align}\label{hoshi-6}
\sum_{n\leq x} \left(\sum_{p\leq z}f_{p}(n)\right)^{k}
=x\left(\sum_{\begin{subarray}{c}p_{1}, \ldots, p_{k}\leq z\\ p_{1}\cdots p_{k}\,:\, \textit{square-full}\end{subarray}}G(p_{1}\cdots p_{k})\right)
 +O\left(2^{k}\pi(z)^{k}\right).
\end{align}

Next, let $k\geq 2$, and we shall divide the  sum of $G(p_{1}\cdots p_{k})$ on the right-hand side of (\ref{hoshi-6}):
\begin{align}
&\sum_{\begin{subarray}{c}p_{1}, \ldots, p_{k}\leq z\\ p_{1}\cdots p_{k}\,:\, \textit{square-full}\end{subarray}}G(p_{1}\cdots p_{k})\nonumber \\
&=\sum_{1\leq s\leq \frac{k}{2}}\sum_{q_{1}<\cdots <q_{s} \leq z}
 \sum_{\begin{subarray}{c}\alpha_{1}+\cdots +\alpha_{s}=k\\
                          \alpha_{1}\geq 2, \ldots, \alpha_{s}\geq 2\end{subarray}}
         \frac{k!}{\alpha_{1}!\cdots \alpha_{s}!} G\left(q_{1}^{\alpha_{1}}\cdots q_{s}^{\alpha_{s}}\right) \nonumber \\
&=
\begin{cases}
{\displaystyle I_{\frac{k}{2}}(k) +\sum_{s\leq \frac{k}{2}-1}I_{s}(k)}     & (k\geq 2: even)\\
{\displaystyle I_{\frac{k-1}{2}}(k) +\sum_{s\leq \frac{k-1}{2}-1}I_{s}(k)} & (k\geq 3: odd)
\end{cases}, \label{I-wa}
\end{align} 
say.

Henceforth, we shall investigate the right-hand side of (\ref{I-wa}) for odd integers $k\geq 3$. 
\bigskip

\noindent{\bf (i) The sum $I_{\frac{k-1}{2}}(k)$.}
Let $k\geq 3$ be any integer.
As for $I_{\frac{k-1}{2}}(k)$, we remark that there are just $\frac{k-1}{2}$ ways of $\frac{k-1}{2}$-tuple $\{\alpha_{1}, \ldots, \alpha_{\frac{k-1}{2}}\}$ of integers
satisfying $\alpha_{1}+\cdots +\alpha_{\frac{k-1}{2}}=k$ and $\alpha_{i}\geq 2$ ($i=1,\ldots, \frac{k-1}{2}$):
\begin{align*}
\begin{cases}
& \alpha_{1}=3, \alpha_{2}=2, \ldots, \alpha_{\frac{k-1}{2}}=2,\\
& \quad \cdots\\
& \alpha_{1}=2, \ldots, \alpha_{\frac{k-1}{2}-1}=2, \alpha_{\frac{k-1}{2}}=3.
\end{cases}
\end{align*}
(If $k=3$, then $\alpha_{1}=3$, only.)
Hence, we see that
\begin{align}\label{a-hoshi-da}
I_{\frac{k-1}{2}}(k)=\frac{k!}{3\cdot 2^{\frac{k-1}{2}}}\sum_{i=1}^{\frac{k-1}{2}}\sum_{q_{1}<\cdots <q_{\frac{k-1}{2}}\leq z}
                    G\left(q_{i}^{3}\right) \prod_{\begin{subarray}{c}j=1\\ j\neq i\end{subarray}}^{\frac{k-1}{2}} G\left(q_{j}^{2}\right),
\end{align}
here, (ii) of Lemma \ref{G(r)-lemma}:
\begin{align*}
 G\left(q_{i}^{3}\right)=\frac{1}{q_{i}}\left(1-\frac{1}{q_{i}}\right)\left(1-\frac{2}{q_{i}}\right) \geq 0,
\quad G\left(q_{j}^{2}\right) =\frac{1}{q_{j}}\left(1-\frac{1}{q_{j}}\right) >0.
\end{align*}
Noting that $G\left(p^{2}\right)> G\left(p^{3}\right)\geq 0$ (for any $p$) we have
\begin{align}\label{hachi-odd}
I_{\frac{k-1}{2}}(k)
\begin{cases}
{\displaystyle < \tilde{C}_{k}\sum_{\begin{subarray}{c}q_{1},\ldots, q_{\frac{k-1}{2}} \leq z\\ \textit{distinct}\end{subarray}} \prod_{j=1}^{\frac{k-1}{2}}G(q_{j}^{2}) } & \\
{\displaystyle > \tilde{C}_{k}\sum_{\begin{subarray}{c}q_{1},\ldots, q_{\frac{k-1}{2}} \leq z\\ \textit{distinct}\end{subarray}} \prod_{j=1}^{\frac{k-1}{2}}G(q_{j}^{3}) }&
\end{cases},
\end{align}
where $\tilde{C}_{k}$ is the constant defined in (\ref{teigi-C-j-tilde}). By use of $\sum_{p\leq z}\frac{1}{p}=\log\log z +O(1)$, we have
\begin{align}\label{hachi-odd-1}
I_{\frac{k-1}{2}}(k) &< \tilde{C}_{k}\left(\sum_{p\leq z}\frac{1}{p}\left(1-\frac{1}{p}\right)\right)^{\frac{k-1}{2}}\nonumber \\
                  &= \tilde{C}_{k}(\log\log z)^{\frac{k-1}{2}} \left(1+O\left(\frac{1}{\log\log z}\right)\right)^{\frac{k-1}{2}}.
\end{align}
To find a lower bound of $I_{\frac{k-1}{2}}(k)$ we note that $G\left(p^{3}\right)$ is monotonic decreasing for $p\geq 5$. 
First, we shall restrict $q_{i}\geq 5$ in the latter of (\ref{hachi-odd}), that is,
\begin{align}
I_{\frac{k-1}{2}}(k) &> \tilde{C}_{k} \sum_{\begin{subarray}{c} 5\leq q_{1},\ldots, q_{\frac{k-1}{2}}\leq z \\ \textit{distinct} \end{subarray}}
                       \prod_{j=1}^{\frac{k-1}{2}}G\left(q_{j}^{3}\right)\nonumber \\
                   &= \tilde{C}_{k} \sum_{\begin{subarray}{c} 5\leq q_{1},\ldots, q_{\frac{k-1}{2}-1}\leq z \\ \textit{distinct} \end{subarray}}
                      \prod_{j=1}^{\frac{k-1}{2}-1} G\left(q_{j}^{3}\right)\sum_{\begin{subarray}{c}5\leq p \leq z \\ p\neq q_{i}\, (i=1,\ldots, \frac{k-1}{2}-1)\end{subarray}}
                      G\left(p^{3}\right). \label{hachi-odd-2}
\end{align}
Let $\pi_{l}$ be the $l$th smallest prime. Since $G\left(p^{3}\right)$ is monotonic decreasing for $p\geq 5=\pi_{3}$, we find that
\begin{align*}
\sum_{\begin{subarray}{c}5\leq p \leq z \\ p\neq q_{i}\, (i=1,\ldots, \frac{k-1}{2}-1)\end{subarray}}
                      G\left(p^{3}\right) 
\geq \sum_{\pi_{3 +\frac{k-3}{2}}\leq t \leq z} G\left(t^{3}\right),
\end{align*}
where $t$ denotes a prime. Repeating this manner on the right-hand side of (\ref{hachi-odd-2}), we obtain
\begin{align*}
I_{\frac{k-1}{2}}(k) &> \tilde{C}_{k} \left(\sum_{\begin{subarray}{c}\pi_{3+\frac{k-3}{2}} \leq p \leq z\end{subarray}}G\left(p^{3}\right)\right)^{\frac{k-1}{2}}\\
                  &\geq \tilde{C}_{k}\left(\sum_{p\leq z}\frac{1}{p}-\sum_{p\leq \pi_{3+\frac{k-3}{2}}}1 +O(1)\right)^{\frac{k-1}{2}} \\
                  &=\tilde{C}_{k} (\log\log z +O(k))^{\frac{k-1}{2}}\\
                  &=\tilde{C}_{k} (\log\log z)^{\frac{k-1}{2}} \left(1+O\left(\frac{k}{\log\log z}\right)\right)^{\frac{k-1}{2}}.
\end{align*}
At this phase, we shall use the assumption $k\leq (\log\log z)^{1/3}$ for (\ref{hachi-odd-1}) and the above. Then we reach
\begin{align}\label{natsuki-2}
I_{\frac{k-1}{2}}(k)=\tilde{C}_{k}(\log\log z)^{\frac{k-1}{2}} +O\left(\tilde{C}_{k}(\log\log z)^{\frac{k-3}{2}}k^{2}\right).
\end{align}
(For $k=3$ we evaluate (\ref{a-hoshi-da}), directly.)\\

\noindent{\bf (ii) The sum $\sum_{s\leq \frac{k-3}{2}}I_{s}(k)$.}
In the case $k=3$, the sum $\sum_{s\leq \frac{k-3}{2}}I_{s}(k)$ in (\ref{I-wa}) is empty. So, let $k\geq 5$ be any odd integer $\leq (\log\log z)^{1/3}$, 
and we evaluate it. By (i) and (iv) of Lemma \ref{G(r)-lemma}, 
we see that $0\leq G(q_{1}^{\alpha_{1}}\cdots q_{s}^{\alpha_{s}})\leq \frac{1}{q_{1}\cdots q_{s}}$ in (\ref{I-wa}). Then, we have
\begin{align}
\sum_{s\leq \frac{k-3}{2}}I_{s}(k)& \leq \sum_{s\leq \frac{k-3}{2}}k!\sum_{q_{1}<\cdots <q_{s}\leq z}\frac{1}{q_{1}\cdots q_{s}} \sum_{\begin{subarray}{c}\alpha_{1}+\cdots +\alpha_{s}=k\\ \alpha_{1}\geq 2, \ldots, \alpha_{s}\geq 2\end{subarray}}\frac{1}{\alpha_{1}!\cdots \alpha_{s}!}\nonumber \\
&\leq \sum_{s\leq \frac{k-3}{2}}\frac{k!}{2^{s}s!}\left(\sum_{p\leq z}\frac{1}{p}\right)^{s} 
  \sum_{\begin{subarray}{c}\alpha_{1}+\cdots +\alpha_{s}=k\\ \alpha_{1}\geq 2, \ldots, \alpha_{s}\geq 2\end{subarray}}1\nonumber \\
&=\sum_{s\leq \frac{k-3}{2}}\frac{k!}{2^{s}s!} \left(\sum_{p\leq z}\frac{1}{p}\right)^{s} \binom{k-s-1}{s-1}. \label{atsui-atsui}
\end{align}
Note that $\binom{k-s-1}{s-1}<\binom{k-s}{s}$ (by $2s <k$), moreover, in (\ref{atsui-atsui}) for any sufficiently large $z$ we see that
\begin{align*}
\left(\sum_{p\leq z}\frac{1}{p}\right)^{s} <3 (\log\log z)^{s}.
\end{align*}
Then
\begin{align}
&(\textit{RHS of } (\ref{atsui-atsui}))\nonumber \\ 
&< 3C_{k}(\log\log z)^{\frac{k-3}{2}} \sum_{s\leq \frac{k-3}{2}} \frac{\Gamma \left(\frac{k}{2}+1\right)2^{\frac{k}{2}-s}}{s!} 
                                                      \frac{(k-s)!}{s!(k-2s)!}\frac{1}{(\log\log z)^{\frac{k-3}{2}-s}}\nonumber \\
&=6\sqrt{2}C_{k}(\log\log z)^{\frac{k-3}{2}}
  \sum_{j=0}^{\frac{k-5}{2}} \frac{\Gamma \left(\frac{k}{2}+1\right)2^{j}}{\left(\frac{k-3}{2}-j\right)!} \frac{\left(\frac{k+3}{2}+j\right)!}{\left(\frac{k-3}{2}-j\right)!} \frac{1}{(2j+3)!}\frac{1}{(\log \log z)^{j}}. \label{atsui-atsui-2}
\end{align}
Recall \cite[p.~58]{Tit} to get
\begin{align}\label{dango-1}
\Gamma \left(\frac{k}{2}+1\right) \asymp \frac{\Gamma \left(\frac{k}{2}+\frac{3}{2}\right)}{\sqrt{k}}.
\end{align}
Further, we easily see that
\begin{align}\label{dango-2}
\frac{\Gamma \left(\frac{k}{2}+\frac{3}{2}\right)2^{j}}{\left(\frac{k-3}{2}-j\right)!}
=\frac{\left(\frac{k^{2}}{4}-\frac{1}{4}\right)\left(\frac{k-3}{2}\right)! 2^{j}}{\left(\frac{k-3}{2}-j\right)!}
\ll k^{2}\cdot k^{j},
\end{align}
and
\begin{align}\label{dango-3}
\frac{\left(\frac{k+3}{2}+j \right)!}{\left(\frac{k-3}{2}-j \right)!} \ll k^{3} \frac{\left(\frac{k-3}{2}+j\right)!}{\left(\frac{k-3}{2}-j\right)!}\ll k^{3}\cdot k^{2j}.
\end{align}
We shall apply (\ref{dango-1}), (\ref{dango-2}), and (\ref{dango-3}) to the right-hand side of (\ref{atsui-atsui-2}), hence
\begin{align*}
(\textit{RHS of } (\ref{atsui-atsui-2})) & \ll \frac{C_{k}(\log\log z)^{\frac{k-3}{2}}}{\sqrt{k}}\sum_{j=0}^{\frac{k-5}{2}}\frac{\left(k^{2}k^{j}\right)\left(k^{3} k^{2j}\right)}{(2j+3)!}
                                               \frac{1}{(\log\log z)^{j}} \\
&=k^{9/2} C_{k} (\log \log z)^{\frac{k-3}{2}} \sum_{j=0}^{\frac{k-5}{2}} \frac{1}{(2j+3)!} \left(\frac{k^{3}}{\log\log z}\right)^{j}  \\
&\ll k^{9/2} C_{k} (\log\log z)^{\frac{k-3}{2}} \quad (\textit{by } k\leq (\log\log z)^{1/3}) \\
&\asymp \tilde{C}_{k}(\log\log z)^{\frac{k-3}{2}}k^{3} \quad (\textit{by } (\ref{teigi-C-j-tilde})). 
\end{align*}
Therefore, we have
\begin{align}
\sum_{s\leq \frac{k-3}{2}}I_{s}(k) \ll   \tilde{C}_{k}(\log\log z)^{\frac{k-3}{2}}k^{3}. \label{al-12-9-19}
\end{align}
 
\bigskip

At last, collecting (\ref{I-wa}), (\ref{natsuki-2}), and (\ref{al-12-9-19}) we obtain that
\begin{align}\label{13-fri}
\sum_{\begin{subarray}{c}p_{1},\ldots, p_{k}\leq z \\ p_{1}\cdots p_{k}\, :\, \textit{square-full}\end{subarray}}G(p_{1}\cdots p_{k})
=\tilde{C}_{k} (\log\log z)^{\frac{k-1}{2}} \left(1+O\left(\frac{k^{3}}{\log\log z}\right)\right)
\end{align}
and substitute (\ref{13-fri}) in (\ref{hoshi-6}) we reach the assertion of  Lemma \ref{Natsuki-1}. 
\end{proof}

Regarding Lemma \ref{Natsuki-1} we mention some remarks which are used in the proof of Theorem \ref{Hatsumi-1}.
\begin{remark}\label{k=1-lemma-natsuki-1-remark}
Although we see that $\sum_{n\leq x}(\sum_{p\leq z}f_{p}(n))=O\left(\pi(z)\right)$ by choosing $k=1$ on (\ref{hoshi-6}),
we also find it more easily by (\ref{hoshi-3}) with $k=1$. Actually, we observe it as follows.
\begin{align}
\sum_{n\leq x}\left(\sum_{p\leq z}f_{p}(n)\right)&=\sum_{p\leq z}\sum_{n\leq x}f_{p}(n)\nonumber \\
&=\sum_{p\leq z}\left(\sum_{\begin{subarray}{c}n\leq x\\ p|n\end{subarray}}\left(1-\frac{1}{p}\right)+\sum_{n\leq x}\left(-\frac{1}{p}\right) 
          +\sum_{\begin{subarray}{c}n\leq x \\ p|n \end{subarray}}\frac{1}{p}\right)\nonumber \\ 
& =\sum_{p\leq z}\left( \left[\frac{x}{p}\right] -\frac{x}{p} +O\left(\frac{1}{p}\right)\right)\nonumber \\
&= O\left(\pi(z)\right). \label{k=1-lemma-natsuki-1}
\end{align}
\end{remark}
\begin{remark}\label{remark-b-Prop-2-1-bun}
Let $k\geq 3$ be odd integers $\leq (\log\log z)^{1/3}$. From the proof of Lemma \ref{Natsuki-1} ((\ref{hoshi-6}), (\ref{I-wa}), (\ref{hachi-odd-1}), (\ref{al-12-9-19})), 
easily we observe that
\begin{align}
\sum_{n\leq x}\left(\sum_{p\leq z}f_{p}(n)\right)^{k} 
& \ll \tilde{C}_{k} x(\log\log z)^{\frac{k-1}{2}}\left(1+O\left(\frac{k^{3}}{(\log\log z)}\right)\right) +2^{k}\pi(z)^{k} \nonumber \\
&\asymp C_{k}x (\log\log z)^{\frac{k-1}{2}}k^{3/2} +2^{k}\pi(z)^{k} \label{remark-b-Prop-2-1}
\end{align}
which is the assertion (b) of Proposition \ref{GS-prop-2}.
\end{remark}
\begin{remark}\label{yakult-1000-919}
First of all, let $k\geq 2$ be even integers $\leq (\log\log z)^{1/3}$. By (a) of Proposition \ref{GS-prop-2} (or a process similar to Remark \ref{remark-b-Prop-2-1-bun})
we see that
\begin{align}\label{remark-a-Prop-2-1}
\sum_{n\leq}\left(\sum_{p\leq z} f_{p}(n)\right)^{k} \ll C_{k} x(\log\log z)^{k/2} +2^{k}\pi(z)^{k}.
\end{align}
Obviously, this is valid for $k=0$.

Now, let $k\geq 0$ be integers $\leq (\log\log z)^{1/3}$. From (\ref{k=1-lemma-natsuki-1}), (\ref{remark-b-Prop-2-1}), and (\ref{remark-a-Prop-2-1})
we have
\begin{align}\label{Remark-Lemma-2-2-Natsuki-2}
\sum_{n\leq x}\left(\sum_{p\leq z}f_{p}(n)\right)^{k}
\ll
\begin{cases} x        & (k=0),\\
              \pi (z)  & (k=1),\\
              C_{k}x(\log\log z)^{k/2} +2^{k}\pi(z)^{k}& (k\geq 2),
\end{cases}
\end{align}
uniformly in $k$, $z$, and $x$. 
\end{remark}
%
\section{Proof of Theorem \ref{Hatsumi-1}.}
We shall prove Theorem \ref{Hatsumi-1} by use of Lemma \ref{Natsuki-1}, (a) of Proposition \ref{GS-prop-2} and (\ref{Remark-Lemma-2-2-Natsuki-2}).

Let $z\geq e^{e}$ be any real number and $\omega_{z}(n)$ arithmetical function defined by (\ref{teigi-omega-z}). Using Mertens' formula 
\begin{align}\label{MertensFormula-2026310}
\sum_{p\leq z}\frac{1}{p} =\log\log z +B +O\left(\frac{1}{\log z}\right),
\end{align}
(\cite[p.~466]{HW} and (\ref{teigi-teisu-B}))
and the function $f_{p}(n)$ defined by (\ref{teigi-f-1}), we have, as that was written in \cite[p.~18, line 3]{GS}
\begin{align}
\omega_{z}(n)&= \sum_{\begin{subarray}{c}p|n \\ p\leq z\end{subarray}} \left(1-\frac{1}{p}+\frac{1}{p}\right) \nonumber \\
             &=\sum_{\begin{subarray}{c}p|n\\ p\leq z\end{subarray}}\left(1-\frac{1}{p}\right) +\sum_{p\leq z}\frac{1}{p}
                -\sum_{\begin{subarray}{c}p\nmid n \\ p\leq z\end{subarray}}\frac{1}{p}\nonumber \\
             &=\sum_{p\leq z}f_{p}(n) +\log\log z +B +O\left(\frac{1}{\log z}\right). \label{ogai}
\end{align}
From this, for any real $x\geq 1$ and any integer $k\geq 1$ we see that
\begin{align}
\sum_{n\leq x}\left(\omega_{z}(n)-\log\log z\right)^{k}
&=\sum_{n\leq x} \left(\sum_{j=0}^{k-2} \binom{k}{j} \{O(1)\}^{k-j} \left(\sum_{p\leq z}f_{p}(n)\right)^{j}\right)\nonumber \\
&\quad +\left(kB+O\left(\frac{k}{\log z}\right)\right) \sum_{n\leq x} \left(\sum_{p\leq z}f_{p}(n)\right)^{k-1}\nonumber \\
&\quad +\sum_{n\leq x}\left(\sum_{p\leq z} f_{p}(n)\right)^{k}. \label{AL-25}  
\end{align}
From now on, we shall assume that
\begin{align}\label{a-hoshi-futatsu-}
2\leq k \leq (\log\log z)^{1/3},\quad \textit{and}\quad e^{e^{k^3}}\leq z \leq x^{1/k}.
\end{align}
Under the setting, we shall prove (b) of Theorem \ref{Hatsumi-1}, first.

\bigskip

\begin{proof}[Proof of (b)]
Let $k$ be odd and assume (\ref{a-hoshi-futatsu-}). 
On the right-hand side of (\ref{AL-25})
by Lemma \ref{Natsuki-1} we have
\begin{align}
\sum_{n\leq x}\left(\sum_{p\leq z}f_{p}(n)\right)^{k}
=\tilde{C}_{k}x (\log\log z)^{\frac{k-1}{2}} \left(1+O\left(\frac{k^{3}}{\log\log z}\right)\right) +O\left(C_{k}x\right),
\label{AL-26}
\end{align}
and by (a) of Proposition \ref{GS-prop-2} we observe that
\begin{align}
& \left(kB+O\left(\frac{k}{\log z}\right)\right) \sum_{n\leq x} \left(\sum_{p\leq z}f_{p}(n)\right)^{k-1} \nonumber \\
&= Bk C_{k}\frac{C_{k-1}}{C_{k}} x (\log\log z)^{\frac{k-1}{2}} +O\left(k^{4}C_{k}\frac{C_{k-1}}{C_{k}}x (\log\log z)^{\frac{k-3}{2}}\right)\nonumber \\
&\quad +O\left(kC_{k} \frac{C_{k-1}}{C_{k}}x \frac{(\log\log z)^{\frac{k-1}{2}}}{\log z}\right) +O\left(kC_{k} \frac{C_{k-1}}{C_{k}}x\right) \nonumber \\
&= B \left(k^{3/2}C_{k}\right) \left(\frac{C_{k-1}}{k^{1/2}C_{k}}\right) x (\log\log z)^{\frac{k-1}{2}} \nonumber\\
&\quad +O\left((k^{3/2}C_{k})x (\log\log z)^{\frac{k-3}{2}} k^{2}\right)
        +O\left((k^{3/2}C_{k}) x \frac{(\log \log z)^{\frac{k-1}{2}}}{k\log z}\right)\nonumber\\
&\quad +O\left(\left(k^{3/2}C_{k}\right)x k^{-1}\right). \label{AL-27}
\end{align}
Note that in (\ref{AL-25}), for $0\leq j \leq k-2$
\begin{align*}
\sum_{n\leq x}\left(\sum_{p\leq z}f_{p}(n)\right)^{j} \ll C_{j} x(\log\log z)^{\frac{j}{2}},
\end{align*}
by (\ref{Remark-Lemma-2-2-Natsuki-2}) in Remark \ref{yakult-1000-919}. Therefore,
\begin{align}
& \sum_{n\leq x}\sum_{j=0}^{k-2}\binom{k}{j}\left\{O(1)\right\}^{k-j}\left(\sum_{p\leq z}f_{p}(n)\right)^{j} \nonumber \\
&\ll \sum_{j=0}^{k-2} \binom{k}{j}c^{k-j}\left|\sum_{n\leq x}\left(\sum_{p\leq z}f_{p}(n)\right)^{j}\right|\quad (\textit{$c>0$ is a constant})\nonumber \\
&\ll x \sum_{j=0}^{k-2} \binom{k}{j}c^{k-j}C_{j} (\log\log z)^{\frac{j}{2}}\nonumber\\
&=C_{k}x(\log\log z)^{\frac{k-2}{2}}\sum_{j=0}^{k-2} \binom{k}{j}c^{k-j}\frac{C_{j}}{C_{k}}\frac{1}{(\log\log z)^{\frac{k-2}{2}-\frac{j}{2}}} \nonumber \\
&\asymp C_{k}x (\log\log z)^{\frac{k-2}{2}}\sum_{j=0}^{\frac{k-2}{2}} 
                                                \binom{k}{j} c^{k-j}
                  \left(\frac{j}{e}\right)^{\frac{j}{2}}\left(\frac{k}{e}\right)^{-\frac{k}{2}}\frac{1}{(\log\log z)^{\frac{k-2}{2}-\frac{j}{2}}} \nonumber\\
&\leq C_{k} x(\log\log z)^{\frac{k-2}{2}} \sum_{j=0}^{k-2} \binom{k}{j} c^{k-j} \left(\frac{e}{k}\right) ^{\frac{k}{2}-\frac{j}{2}}\left(\frac{k}{e}\right)
                                          \left(\frac{k}{e}\right)^{-1}
                                          \frac{1}{(\log\log z)^{\frac{k-2}{2}-\frac{j}{2}}}\nonumber \\
&= C_{k} x (\log\log  z)^{\frac{k-2}{2}} \sum_{j=0}^{k-2} \binom{k}{j} c^{2}(c^{2})^{\frac{k-2-j}{2}}\left(\frac{e}{k}\right)^{\frac{k-2-j}{2}}
                                         \frac{e/k}{(\log\log z)^{\frac{k-2-j}{2}}}. \label{san-sk-1}
\end{align}
Here, note that
\begin{align*}
\binom{k}{j}=\frac{k(k-1)}{(k-j)(k-j-1)} \binom{k-2}{j} \ll k^{2} \binom{k-2}{j}.
\end{align*}
Then, we see that
\begin{align}
&(\textit{RHS of } (\ref{san-sk-1})) \ll C_{k} x(\log\log z)^{\frac{k-2}{2}} k 
                      \sum_{j=0}^{k-2} \binom{k-2}{j}\left(\frac{c^{2}e/k}{\log\log z}\right)^{\frac{k-2-j}{2}}\nonumber\\
&\ll C_{k} x (\log\log z)^{\frac{k-2}{2}} k \left(1+\frac{(c^{2}e)^{1/2}}{(k\log\log z)^{1/2}}\right)^{k-2} \nonumber \\
&\ll C_{k} x (\log\log z)^{\frac{k-2}{2}} k \quad (\textit{by } k\leq (\log\log z)^{1/3})\nonumber \\
&=\left(k^{3/2}{C}_{k}\right) x (\log\log z)^{\frac{k-2}{2}} \frac{1}{\sqrt{k}}. \label{san-sk-2}
\end{align}
Therefore, by (\ref{AL-25}), (\ref{AL-26}), (\ref{AL-27}), and (\ref{san-sk-2}) we have
\begin{align}
&\sum_{n\leq x}(\omega_{z}(n)-\log\log z)^{k}\nonumber \\
&= O\left((k^{3/2})C_{k} x \frac{(\log\log z)^{\frac{k-2}{2}}}{k^{1/2}}\right)\nonumber\\
&\quad +B(k^{3/2}C_{k}) \left(\frac{C_{k-1}}{k^{1/2}C_{k}}\right) x(\log\log z)^{\frac{k-1}{2}}\nonumber \\
&\quad +O\left( (k^{3/2}C_{k}) x (\log\log z)^{\frac{k-3}{2}}k^{2}\right)\nonumber \\
&\quad +O\left((k^{3/2}C_{k})x\frac{(\log\log z)^{\frac{k-1}{2}}}{k\log z}\right)\nonumber \\
&\quad +\tilde{C}_{k}x (\log\log z)^{\frac{k-1}{2}} +O\left(\tilde{C}_{k} x (\log\log z)^{\frac{k-3}{2}}k^{3}\right)\nonumber\\
&=\tilde{C}_{k}x (\log\log z)^{\frac{k-1}{2}} +B(k^{3/2}C_{k})x (\log\log z)^{\frac{k-1}{2}} \frac{C_{k-1}}{k^{1/2}C_{k}}\nonumber\\
&\quad +O\left((k^{3/2}C_{k})x (\log\log z)^{\frac{k-1}{2}}\max \left\{\frac{1}{(k\log\log z)^{1/2}}, \frac{k^{3}}{\log\log z}\right\}\right). \label{yama-umi-kawa}
\end{align}
Lastly, note that the inequality $1/(k\log\log z)^{1/2} \geq k^{3}/\log\log z$ is equivalent to $(\log\log z)^{1/7} \geq k$,
so the first assertion (I) of (b) is deduced, on the other case $(\log\log z)^{1/7} <k\leq (\log\log z)^{1/3}$, note that the inequality
$k^{3}/\log\log z \geq 1/k\, (\asymp C_{k-1}/(k^{1/2}C_{k}))$ is equivalent to $k\geq (\log\log z)^{1/4}$, then the second assertion (II) and 
the third (III) are follows from (\ref{yama-umi-kawa}). 
\end{proof}

We may obtain the assertion (a), without using Lemma \ref{Natsuki-1}.

\begin{proof}[Proof of (a)]
Let $k\geq 2$ be even and assume (\ref{a-hoshi-futatsu-}). 
In (\ref{AL-25}), we shall apply (a) of Proposition \ref{GS-prop-2}, then
\begin{align*}
\sum_{n\leq x}\left(\sum_{p\leq z}f_{p}(n)\right)^{k} =C_{k}x(\log\log z)^{\frac{k}{2}} \left(1+O\left(\frac{k^{3}}{\log\log z}\right)\right) +O\left(C_{k}x\right).
\end{align*} 
Moreover, by (\ref{remark-b-Prop-2-1}) in Remark \ref{remark-b-Prop-2-1-bun} (here $k-1$ is odd)
\begin{align*}
& \left(kB+O\left(\frac{k}{\log z}\right)\right)\sum_{n\leq x}\left(\sum_{p\leq z}f_{p}(n)\right)^{k-1}\\
&\ll k(C_{k-1}x(\log\log z)^{\frac{k-1}{2}}k^{3/2}+C_{k-1}x) 
\ll C_{k} x (\log\log z)^{\frac{k}{2}}\frac{k^{2}}{\log\log z}
\end{align*}
and by (\ref{Remark-Lemma-2-2-Natsuki-2}) of Remark \ref{yakult-1000-919}
\begin{align*}
& \sum_{n\leq x}\left(\sum_{j=0}^{k-2}\binom{k}{j} \{O(1)\}^{k-j}\right) \left(\sum_{p\leq z}f_{p}(n)\right)^{j}
 \ll C_{k}x(\log\log z)^{\frac{k-2}{2}}k.
\end{align*}
Then we obtain the assertion (a), immediately.
\end{proof}

From Theorem \ref{Hatsumi-1} we have the following corollary.
\begin{corollary}\label{k-jyo-corocoro}
Let $x\geq 1$ and $z\geq 1$ be sufficiently large.
For fixed integer $k\geq 2$ and the restriction $z\leq x^{1/k}$, we have
\begin{align*}
&\sum_{n\leq x}(\omega_{z}(n)-\log\log z)^{k}\\
&=\begin{cases}
{\displaystyle C_{k}x (\log\log z)^{\frac{k}{2}} +O\left(x(\log\log z)^{\frac{k-2}{2}}\right)} & (k\geq 2: \textit{even})\\
{\displaystyle (\tilde{C}_{k} +BkC_{k-1})x (\log\log z)^{\frac{k-1}{2}} +O\left(x(\log\log z)^{\frac{k-3}{2}}\right)} & (k\geq 3: \textit{odd})
\end{cases}.
\end{align*}
\end{corollary}

However, in the cases $k=2,\ldots, 6$ we can obtain more precise asymptotic formulas for $\sum_{n\leq x}(\omega_{z}(n)-\log\log z)^{k}$. To derive that, 
we shall introduce a notation.

\begin{definition}\label{def-GPA}
Let $A=\{\alpha_{1},\ldots, \alpha_{l}\}$ be a finite set of integers $\alpha_{i}\geq 2$. For any prime $p$ and the function $G(\cdot)$ defined by (\ref{916-def}),
we define $G(p,A)$ as
\begin{align*}
G(p,A):=G\left(p^{\alpha_{1}}\right)\cdots G\left(p^{\alpha_{l}}\right).
\end{align*}
By Lemma \ref{G(r)-lemma} and (\ref{MertensFormula-2026310}), we can put constants $B(\alpha)$ and $D(A)$ as
\begin{align}\label{pelikan-310}
\sum_{n\leq z}G(p,A)=
\begin{cases}\log\log z + B(\alpha) +O\left(\frac{1}{\log z}\right) & (A=\{\alpha\})\\
             D(A)+O\left(\frac{1}{z^{\sharp A -1}}\right)            & (\sharp A \geq 2)
\end{cases}.
\end{align} 
We may write $D(\alpha_{1},\ldots, \alpha_{l})$ instead of $D(A)$.
\end{definition}

\begin{theorem}\label{Sakai-Haruka-Theorem}
Keep the notation above, and $B$ be the positive constant defined in (\ref{teigi-teisu-B}).
For sufficiently large $x\geq 1$ and $z\geq 1$, we have the following formulas.
\begin{enumerate}
\item[\rm (a)] Under the restriction $z\leq x^{1/2}$, then
\begin{align*}
\sum_{n\leq x}(\omega_{z}(n)-\log\log z)^{2}=x\log\log z +\left(B(2)+B^{2}\right)x+O\left(\frac{x}{\log z}\right).
\end{align*}
\item[\rm (b)] Under the restriction $z\leq x^{1/3}$, then
\begin{align*}
&\sum_{n\leq x}(\omega_{z}(n)-\log\log z)^{3}\\
&= (1+3B)x\log\log z +\left(B(3)+3B(2)B+B^{3}\right)x +O\left(\frac{x\log\log z}{\log z}\right).
\end{align*}
\item[\rm (c)] Under the restriction $z\leq x^{1/4}$, then
\begin{align*}
&\sum_{n\leq x}(\omega_{z}(n)-\log\log z)^{4}\\
&=3x(\log\log z)^{2} +\left(6B(2)+4B+6B^{2}+1\right)x\log\log z\\
&\quad +\left(3B(2)^{2}-3D(2,2)+B(4)+4B(3)B+6B(2)B^{2}+B^{4}\right)x\\
&\quad +O\left(\frac{x\log\log z}{\log z}\right).
\end{align*}
\item[\rm (d)] Under the restriction $z\leq x^{1/5}$, then
\begin{align*}
&\sum_{n\leq x}(\omega_{z}(n)-\log\log z)^{5}\\
&=(10+15B)x(\log\log z)^{2}\\
&\quad +\left(10(B(3)+B(2))+1 +5(6B(2)+1)B+10B^{2}+10B^{3}\right)x\log\log z\\
&\quad +\left(10(B(3)B(2)-D(3,2))+B(5)\right.\\
&\quad\quad  \left.+5(3B(2)^{2}-3D(2,2)+B(4))B+10B(3)B^{2}+10B(2)B^{3}+5B^{5}\right)x\\
&\quad +O\left(\frac{x(\log\log z)^{2}}{\log z}\right).
\end{align*}
\item[\rm (e)] Under the restriction $z\leq x^{1/6}$, then
\begin{align*}
&\sum_{n\leq x}(\omega_{z}(n)-\log\log z)^{6}\\
&= 15 x (\log\log z)^{3}+ (45B(2)+25 +60B+45B^{2})x(\log\log x)^{2}\\
&\quad +\left(45B(2)^{2}-45D(2,2)+15(B(2)+B(4))+20B(3)+1\right.\\
&\quad +\left. 6(10(B(3)+B(2))+1)B+15(6B(2)+1)B^{2}+20B^{3}+15B^{4}\right)\times \\
&\quad\quad  \times x\log \log z\\
&\quad +\left(15B(2)^{3}-45B(2)D(2,2)+30D(2,2,2)\right.\\
&\quad\quad  +15(B(2)B(4)-D(4,2))+10\left(B(3)^{2}-D(3,3)\right)+B(6)\\
&\quad\quad  +6(10(B(2)B(3)-D(3,2))+B(5))B\\
&\quad\quad  +15(3B(2)^{2}-3D(2,2)+B(4))B^{2}\\
&\quad\quad  +\left. 20B(3)B^{3}+15B(2)B^{4}+B^{6} \right)x +O\left(\frac{x(\log\log z)^{2}}{\log z}\right).
\end{align*}
\end{enumerate}
\end{theorem}
\begin{proof}
Let $k\geq 2$ be a fixed integer. By (\ref{ogai}), we have
\begin{align}\label{23456-formula}
\sum_{n\leq x}(\omega_{z}(n)-\log\log z)^{k}=\sum_{l=0}^{k}\binom{k}{l}\sum_{n\leq x}\left(\sum_{p\leq z}f_{p}(n)\right)^{l}B^{k-l}+\sum_{n\leq x}S(k),
\end{align}
where
\begin{align*}
S(k)=\sum_{l=1}^{k}\binom{k}{l}\left(\sum_{p\leq z}f_{p}(n)+B\right)^{k-l}O\left(\frac{1}{(\log z)^{l}}\right).
\end{align*}
Now, assume that $z\leq x^{1/k}$. By (\ref{k=1-lemma-natsuki-1}), (\ref{remark-b-Prop-2-1}), and (\ref{remark-a-Prop-2-1}) we see that
\begin{align}\label{23456-error}
\sum_{n\leq x}S(k) &\ll \frac{1}{\log z}\left| \sum_{n\leq x}\left(\sum_{p\leq z}f_{p}(n)\right)^{k-1}\right|\nonumber \\
&\ll \frac{1}{\log z} \times \begin{cases} x(\log\log z)^{\frac{k-2}{2}}, & (k\geq 2, \textit{even}),\\
                    x(\log\log z)^{\frac{k-1}{2}}, & (k\geq 3, \textit{odd}).
\end{cases}
\end{align}
Moreover, we shall derive formulas for $\sum_{n\leq x}(\sum_{p\leq z}f_{p}(n))^{k}$:
\begin{align}
\sum_{n\leq x}\left(\sum_{p\leq z}f_{p}(n)\right)^{2} 
=x\log\log z +B(2)x +O\left(\frac{x}{\log z}\right), \quad (z\leq x^{1/2}), \label{S-f2-sum}
\end{align}

\begin{align}
\sum_{n\leq x}\left(\sum_{p\leq z}f_{p}(n)\right)^{3}
=x\log\log z +B(3)x +O\left(\frac{x}{\log z}\right), \quad (z\leq x^{1/3}), \label{S-f3-sum}
\end{align}

\begin{align}
&\sum_{n\leq x}\left(\sum_{p\leq z}f_{p}(n)\right)^{4}\nonumber \\
&=3x (\log\log z)^{2} +(6B(2)+1)x\log\log z\nonumber \\
&\quad +(3B(2)^{2}-3D(2,2)+B(4))x +O\left(\frac{x\log\log z}{\log z}\right), \quad (z\leq x^{1/4}), \label{S-f4-sum}
\end{align}

\begin{align}
&\sum_{n\leq x}\left(\sum_{p\leq z}f_{p}(n)\right)^{5}\nonumber \\
&=10x (\log\log z)^{2} +(10(B(3)+B(2))+1)x\log\log z \nonumber \\
&\quad +(10(B(3)B(2)-D(3,2))+B(5))x +O\left(\frac{x\log\log z}{\log z}\right), \label{S-f5-sum}
\end{align}
($z\leq x^{1/5}$), and 
\begin{align}
&\sum_{n\leq x}\left(\sum_{p\leq z}f_{p}(n)\right)^{6}\nonumber \\
&=15x (\log\log z)^{3} +(45B(2)+25)x(\log\log z)^{2}\nonumber \\
&\quad +(45B(2)^{2}-45D(2,2)+15(B(2)+B(4))+20B(3)+1)x\log\log z\nonumber \\
&\quad +(15B(2)^{3}-45D(2,2)B(2)+30D(2,2,2)\nonumber \\
&\quad\quad   +15(B(4)B(2)-D(4,2))+10\left(B(3)^{2}-D(3,3)\right)+B(6))x\nonumber\\
&\quad +O\left(\frac{x(\log\log z)^{2}}{\log z}\right), \label{S-f6-sum}
\end{align}
($z\leq x^{1/6}$).

To show (\ref{S-f2-sum})--(\ref{S-f6-sum}) we shall go back to (\ref{hoshi-6}) and (\ref{I-wa}). As for $k=2,3$, by $\pi(z)\ll \frac{z}{\log z}$ and $z\leq x^{1/k}$
we see that
\begin{align*}
&\sum_{n\leq x}\left(\sum_{p\leq z}f_{p}(n)\right)^{2}=xI_{1}(2) +O\left(\frac{x}{(\log z)^{2}}\right),\quad (z\leq x^{1/2}),\\
&\sum_{n\leq x}\left(\sum_{p\leq z}f_{p}(n)\right)^{3}=xI_{1}(3) +O\left(\frac{x}{(\log z)^{3}}\right),\quad (z\leq x^{1/3}).
\end{align*}
By (\ref{pelikan-310}), since $I_{1}(2)=\sum_{p\leq z}G(p^{2})=\log\log z+B(2)+O\left(\frac{1}{\log z}\right)$ and 
$I_{1}(3)=\sum_{p\leq z}G(p^{3})=\log\log z +B(3)+O\left(\frac{1}{\log z}\right)$, we get (\ref{S-f2-sum}) and (\ref{S-f3-sum}), at once.

As for $k=4$, first we see that
\begin{align}\label{f4-sum-step-1}
\sum_{n\leq x}\left(\sum_{p\leq z}f_{p}(n)\right)^{4}=x(I_{2}(4)+I_{1}(4))+O\left(\frac{x}{(\log z)^{4}}\right).
\end{align}
Using Lemma \ref{G(r)-lemma} and the formula  (\ref{pelikan-310}) in Definition \ref{def-GPA} we observe that
\begin{align}\label{f4-sum-step-2}
I_{2}(4) 
&=6\sum_{q_{1}<q_{2}\leq z}G\left(q_{1}^{2}\right)G\left(q_{2}^{2}\right)\nonumber\\
        &=3 \left(\left(\sum_{q_{1}\leq z}G\left(q_{1}^{2}\right)\right)\left(\sum_{q_{2}\leq z}G\left(q_{2}^{2}\right)\right)-\sum_{q\leq z}G\left(q^{2}\right)G(q^{2})\right)\nonumber \\
        &=3 (\log\log z)^{2} +6B(2)\log\log z +3B(2)^{2}-3D(2,2) \nonumber \\
        &\quad +O\left(\frac{\log\log z}{\log z}\right).
\end{align}
And, easily we see
\begin{align}\label{f4-sum-step-3}
I_{1}(3)=\sum_{q_{1}\leq z}G(q_{1}^{4})=\log\log z +B(4)+O\left(\frac{1}{\log z}\right).
\end{align}
Combining (\ref{f4-sum-step-1})--(\ref{f4-sum-step-3}) we get (\ref{S-f4-sum}).

As for $k=5$, we have
\begin{align*}
\sum_{n\leq x}\left(\sum_{p\leq z}f_{p}(n)\right)^{5}=x(I_{2}(5)+I_{1}(5))+O\left(\frac{x}{(\log z)^{5}}\right),
\end{align*}
\begin{align*}
I_{2}(5)&=10\sum_{\begin{subarray}{c}q_{1}, q_{2}\leq z\\ \textit{distinct}\end{subarray}}G\left(q_{1}^{3}\right)G\left(q_{2}^{2}\right)\\
        &=10(\log\log z)^{2} +10(B(3)+B(2))\log\log z \\
        &\quad   +10(B(3)B(2)-D(3,2)) +O\left(\frac{\log\log z}{\log z}\right),
\end{align*}
and
\begin{align*}
I_{1}(5)=\sum_{q\leq z}G(q^{5})=\log\log z +B(5)+O\left(\frac{1}{\log z}\right),
\end{align*}
therefore we get (\ref{S-f5-sum}).

As for $k=6$, we have
\begin{align*}
\sum_{n\leq x}\left(\sum_{p\leq z}f_{p}(n)\right)^{6}=x(I_{3}(6)+I_{2}(6)+I_{1}(6))+O\left(\frac{x}{(\log z)^{6}}\right),
\end{align*}

\begin{align*}
I_{3}(6)
&=90 \sum_{q_{1}<q_{2}< q_{3}\leq z}G\left(q_{1}^{2}\right)G\left(q_{2}^{2}\right)G\left(q_{3}^{2}\right)\\
&=15\sum_{\begin{subarray}{c}q_{1},q_{2},q_{3}\leq z\\ \textit{distinct}\end{subarray}} G\left(q_{1}^{2}\right)G\left(q_{2}^{2}\right)G\left(q_{3}^{2}\right)\\
&=15 \left(\left(\sum_{p\leq z}G(p^{2})\right)^{3} -3\sum_{\begin{subarray}{c}p, q\leq z\\ \textit{distinct}\end{subarray}}\left(G(p^{2})\right)^{2}G(q^{2})
               -\sum_{p\leq z}\left(G(p^{2})\right)^{3}\right)\\
&=15(\log\log z)^{3}\\
&\quad +45B(2)(\log\log z)^{2}+ (45B(2)^{2}-45D(2,2))\log\log z\\
&\quad +15B(2)^{3}-45D(2,2)B(2)+30D(2,2,2)+O\left(\frac{(\log\log z)^{2}}{\log z}\right),
\end{align*}

\begin{align*}
I_{2}(6)&=15\sum_{q_{1}<q_{2}\leq z}\left(G(q_{1}^{4})G\left(q_{2}^{2}\right)+G(q_{1}^{2})G\left(q_{2}^{4}\right)\right)+20\sum_{q_{1}<q_{2}\leq z}G\left(q_{1}^{3}\right)G\left(q_{2}^{3}\right)\\
&=15\sum_{\begin{subarray}{c}q_{1}, q_{2}\leq z\\ \textit{distinct}\end{subarray}}G\left(q_{1}^{4}\right)G\left(q_{2}^{2}\right)
   +10\sum_{\begin{subarray}{c}q_{1}, q_{2}\leq z\\ \textit{distinct}\end{subarray}}G\left(q_{1}^{3}\right)G\left(q_{2}^{3}\right)\\    
        &=15 (\log\log z)^{2}+15 (B(4)+B(2))\log\log z \\
&\quad +15 (B(4)B(2)-D(4,2)) +O\left(\frac{\log\log z}{\log z}\right)\\
&\quad +10 (\log\log z)^{2} +20B(3)(\log\log z) +10\left(B(3)^{2}-D(3,2)\right)\\
&\quad +O\left(\frac{\log\log z}{\log z}\right)\\
&= 25(\log\log z)^{2}+\left(15(B(4)+B(2)) +20B(3)\right)\log\log z \\
&\quad +15(B(4)B(2)-D(4,2))+10(B(3)^{2}-D(3,2))\\
&\quad +O\left(\frac{\log\log z}{\log z}\right) 
\end{align*}

\begin{align*}
I_{1}(6)=\sum_{q\leq z}G(q^{6}) =\log\log z +B(6) +O\left(\frac{1}{\log z}\right),
\end{align*}
then, we reach (\ref{S-f6-sum}).

Finally, for each $k=2,3,4,5,6$ we apply (\ref{S-f2-sum})--(\ref{S-f6-sum}) to (\ref{23456-formula}) with (\ref{23456-error}),
then we get the assertions of Theorem  \ref{Sakai-Haruka-Theorem}.
\end{proof}

\begin{remark}\label{remark-akademori} It is easily to get the assertion on Theorem \ref{Hatsumi-1} for $k=1$ with $2\leq z\leq x$. By (\ref{ogai}) and (\ref{k=1-lemma-natsuki-1}) 
(or (\ref{Remark-Lemma-2-2-Natsuki-2}), $k=1$)
we have
\begin{align*}
\sum_{n\leq x}(\omega_{z}(n)-\log\log z)&=\sum_{n\leq x}\left(\sum_{p\leq z}f_{p}(n)\right) +Bx +O\left(\frac{x}{\log z}\right)\\
&=Bx +O\left(\frac{x}{\log z}\right).
\end{align*}
\end{remark}
%
%
%
At the last of this section we shall now improve Corollary \ref{k-jyo-corocoro}. 
\begin{theorem}\label{K-theorem}
Keep the notation the above.
Let $k\geq 2$ be a fixed integer and $z\leq x^{1/k}$. For sufficiently large $x\geq e^{e}$ and $z\geq e^{e}$, we have
\begin{align}
&\sum_{n\leq x}(\omega_{z}(n)-\log\log z)^{k}\nonumber \\
&=x\sum_{j=0}^{\left[\frac{k}{2}\right]}a_{j}(\log\log z)^{j} +
\begin{cases} O\left(\frac{x(\log\log z)^{\frac{k-2}{2}}}{\log z}\right), & (k\geq 2,\ even )\\
O\left(\frac{x(\log\log z)^{\frac{k-1}{2}}}{\log z}\right)                & (k\geq 3,\ odd )
\end{cases}\label{K-theorem-12}
\end{align}
where $a_{j}$ are constants, especially
\begin{align}
a_{\left[\frac{k}{2}\right]}=
\begin{cases}
C_{k} & (k\geq 2,\ even)\\
\tilde{C}_{k} +Bk C_{k-1} & (k\geq 3,\ odd)
\end{cases}. \label{K-theorem-22}
\end{align}
\end{theorem}

First we shall show the following proposition.
\begin{proposition}\label{K-prop-1}
We fix an integer $k\geq 2$ and keep the notation of $G(p,A)$ in Definition \ref{def-GPA}.
Let $L$ be any integer $\geq 2$. For $L$ many nonempty finite set $A_{1}, \ldots, A_{L}$ which are sets of integers $\geq 2$, respecitively,
we have for sufficiently large $z\geq e^{e}$
\begin{align*}
\sum_{\begin{subarray}{c}p_{1}, \ldots, p_{L}\leq z\\ distinct\end{subarray}}G(p_{1}, A_{1})\cdots G(p_{L}, A_{L})
=\sum_{j=0}^{L}\mathcal{L}_{j}(\log\log z)^{j} +O\left(\frac{(\log\log z)^{L-1}}{\log z}\right),
\end{align*}
where $\mathcal{L}_{j}$ ($j=1,\ldots, L$) are constants. Especially, if
$\sharp A_{1}=\cdots =\sharp A_{L}=1$, then $\mathcal{L}_{L}=1$.
\end{proposition}
\begin{proof}
We use an induction. When $L=2$, we see that
\begin{align*}
&\left(\sum_{p_{1}\leq z}G(p_{1}, A_{1})\right)\left(\sum_{p_{2}\leq z}G(p_{2}, A_{2})\right)\\
&=\sum_{p\leq z}G(p, A_{1}\sqcup A_{2})+\sum_{\begin{subarray}{c}p_{1},p_{2}\leq z \\ distinct\end{subarray}}G(p_{1}, A_{1})G(p_{2}, A_{2}).
\end{align*}
On the other hand, by formulas stated in Definition \ref{def-GPA} we observe that
\begin{align*}
&\left(\sum_{p_{1}\leq z}G(p_{1}, A_{1})\right)\left(\sum_{p_{2}\leq z}G(p_{2}, A_{2})\right)\\
&=\begin{cases}
(\log\log z)^{2} +a_{1}\log\log z +a_{0} +O\left(\frac{\log\log z}{\log z}\right) & (\sharp A_{1}=\sharp A_{2}=1),\\
a_{1}^{\prime}\log\log z +a_{0}^{\prime} +O\left(\frac{1}{\log z}\right)          & (\sharp A_{1}=1, \sharp A_{2}\geq 2,\\ 
                                                                                  &\quad or\ \sharp A_{1}\geq 2, \sharp A_{2}=1),\\
a_{0}^{\prime\prime} + O\left(\frac{1}{z}\right) & (\sharp A_{1}\geq 2,\ \sharp A_{2}\geq 2),
\end{cases}
\end{align*}
and
\begin{align*}
\sum_{p\leq z}G(p,A_{1} \sqcup A_{2}) =D(A_{1}\sqcup A_{2}) +O\left(\frac{1}{z}\right).
\end{align*}
Hence, easily we get the assertion for $L=2$.

For $L=2,3,\ldots, M$, we shall assume the assertions hold. When $L=M+1$, we consider
\begin{align*}
\sum_{\begin{subarray}{c}p_{1}, \ldots, p_{M+1}\leq z\\ distinct\end{subarray}}G(p_{1}, A_{1})\cdots G(p_{M+1}, A_{M+1}).
\end{align*}
As for $A:=A_{1}\sqcup \cdots \sqcup A_{M+1}$ we shall split this into $j$ many sets, in which each set is a disjoint union of $A_{i}$.
Denote by $S(m,n)$ the Stirling number of the second kind, that is, the number of ways to divide $m$-elements into nonempty $n$-classes. 
We shall write $S(M+1,j)$-partitions of $A$ as
\begin{align*}
A&=\tilde{A}_{j,1,1}\sqcup \cdots \sqcup \tilde{A}_{j,1,j}\\
 &=\tilde{A}_{j,2,1}\sqcup \cdots \sqcup \tilde{A}_{j,2,j}\\
 &\cdots \\
&=\tilde{A}_{j,S(M+1,j),1}\sqcup \cdots \sqcup \tilde{A}_{j,S(M+1,j),j}.
\end{align*} 
We observe that 
\begin{align*}
&\left(\sum_{p_{1}\leq z} G(p_{1},A_{1})\right)\cdots \left(\sum_{p_{M+1}\leq z} G(p_{M+1},A_{M+1})\right)\\
&=\sum_{j=1}^{M+1}\sum_{\begin{subarray}{c}\{\tilde{A}_{j,l,1}, \ldots, \tilde{A}_{j,l,j}\}\\ l=1,\ldots, S(M+1,j)\end{subarray}}
\sum_{\begin{subarray}{c}p_{1},\ldots, p_{j}\leq z\\ distinct\end{subarray}}G(p_{1}, \tilde{A}_{j,l,1})\cdots G(p_{j}, \tilde{A}_{j,l,j})\\
&=\sum_{\begin{subarray}{c}p_{1},\ldots, p_{M+1}\leq z\\ distinct \end{subarray}}G(p_{1}, A_{1})\cdots G(p_{M+1}, A_{M+1})\\
&\quad +\sum_{j=1}^{M}\sum_{\begin{subarray}{c}\{\tilde{A}_{j,l,1}, \ldots, \tilde{A}_{j,l,j}\}\\ l=1,\ldots, S(M+1,j)\end{subarray}}
\sum_{\begin{subarray}{c}p_{1},\ldots, p_{j}\leq z\\ distinct\end{subarray}}G(p_{1}, \tilde{A}_{j,l,1})\cdots G(p_{j}, \tilde{A}_{j,l,j})
\end{align*}
and by the assumption and formulas in Definition \ref{def-GPA}, we find that
\begin{align*}
&\sum_{\begin{subarray}{c}p_{1},\ldots, p_{M+1}\leq z\\ distinct \end{subarray}}G(p_{1}, A_{1})\cdots G(p_{M+1}, A_{M+1})\\
&=\begin{cases}
(\log\log z)^{M+1} +\sum_{j=0}^{M}b_{j}(\log\log z)^{j} +O\left(\frac{(\log\log z)^{M}}{\log z}\right), & (\forall i, \sharp A_{i}=1)\\
\sum_{j=0}^{M}b_{j}^{\prime}(\log\log z)^{j} +O\left(\frac{(\log\log z)^{M-1}}{\log z}\right), & (otherwise)
\end{cases}\\
&=\sum_{j=0}^{M+1}\mathcal{L}_{j}(\log\log z)^{j}+O\left(\frac{(\log\log z)^{M}}{\log z}\right).
\end{align*}
Then, we complete the proof.
\end{proof}

Next, we shall connect Proposition \ref{K-prop-1} to $I_{s}(k)$ in (\ref{I-wa}). So as to do, we introduce
some notation for $I_{s}(k)$.

\begin{definition}\label{def-NKS}
As for $I_{s}(k)$ in (\ref{I-wa}), we set
\begin{align*}
N(k,s):=\{(\alpha_{1}, \ldots, \alpha_{s})\, |\, \alpha_{1}+\cdots +\alpha_{s}=k,\ \forall \alpha_{i}\geq 2\}.
\end{align*}
Note that since $s\leq k/2$ in (\ref{I-wa}), we see that $\sharp N(k,s)=\binom{k-s-1}{s-1}\geq 1$. We define an equivalence relation $\simeq_{1}$ on $N(k,s)$:\\

If $(\alpha_{1}, \ldots, \alpha_{s})$ and $(\alpha_{1}^{\prime}, \ldots, \alpha_{s}^{\prime})$ in $N(k,s)$ satisfy that 
$\{\alpha_{1}, \ldots, \alpha_{s}\}=\{\alpha_{1}^{\prime}, \ldots, \alpha_{s}^{\prime}\}$ (ignore the order of elements), then
we write $(\alpha_{1}, \ldots, \alpha_{s})\simeq_{1} (\alpha_{1}^{\prime}, \ldots, \alpha_{s}^{\prime})$. 

By this relation we obtain the partion of $N(k,s)$ as
\begin{align*}
N(k,s)=\bigsqcup_{l=1}^{n(k,s)}N_{l}(k,s),
\end{align*} 
where $N_{l}(k,s)$ denote equivalence classes, and $n(k,s)$ denotes the number of them by the relation $\simeq_{1}$.

For each class $N_{l}(k,s)$ we put
\begin{align*}
c(N_{l}(k,s)):=\frac{k!}{\alpha_{1}!\cdots \alpha_{s}!}.
\end{align*}  
Moreover, for a representative $(\alpha_{1},\ldots, \alpha_{s})$ of a class $N_{l}(k,s)$, we write all distinct $\alpha_{i}$ as
$\beta_{1}, \ldots, \beta_{d}$,
and the multiplicity of $\beta_{j}$ as
\begin{align*}
m(\beta_{j}):=\sum_{\begin{subarray}{c}\alpha_{i}\\ \alpha_{i}=\beta_{j}\\ ((\alpha_{1}, \ldots, \alpha_{s}): \textit{a representative of $N_{l}(k,s)$})\end{subarray}}1.
\end{align*}
\end{definition}

\begin{proposition}\label{K-prop-2}
Keep the notation the above. For any $I_{s}(k)$ in (\ref{I-wa}), we have
\begin{align*}
I_{s}(k)=\sum_{l=1}^{n(k,s)}\frac{c(N_{l}(k,s))}{m(\beta_{1})!\cdots m(\beta_{d})!}
\hspace{-1cm}\sum_{\begin{subarray}{c}q_{1}, \ldots, q_{s}\leq z\\ distinct \\ ((\alpha_{1},\ldots, \alpha_{s}): \textit{a representative of $N_{l}(k,s)$})\end{subarray}}
\hspace{-1cm}G(q_{1}^{\alpha_{1}})\cdots G(q_{s}^{\alpha_{s}}).
\end{align*}
\end{proposition}
\begin{proof}
As for $I_{s}(k)$ in (\ref{I-wa}), using the Definition \ref{def-NKS} we observe that
\begin{align}
I_{s}(k)&=\sum_{q_{1}<\cdots <q_{s}\leq z}\sum_{l=1}^{n(k,s)}\sum_{\forall (\alpha_{1},\ldots, \alpha_{s})\in N_{l}(k,s)}G(q_{1}^{\alpha_{1}})\cdots G(q_{s}^{\alpha_{s}})\nonumber\\
&= \sum_{l=1}^{n(k,s)}c(N_{l}(k,s))\sum_{\forall (\alpha_{1},\ldots, \alpha_{s})\in N_{l}(k,s)} \sum_{q_{1}<\cdots <q_{s}\leq z}G(q_{1}^{\alpha_{1}})\cdots G(q_{s}^{\alpha_{s}}).
\label{NKS-wa-1}
\end{align} 
Next, we fix a representative $(\alpha_{1},\ldots, \alpha_{s})$ in each class $N_{l}(k,s)$ and show that
for each class $N_{l}(k,s)$
\begin{align}
&\sum_{\forall (\alpha_{1},\ldots, \alpha_{s})\in N_{l}(k,s)} \sum_{q_{1}<\cdots <q_{s}\leq z}G(q_{1}^{\alpha_{1}})\cdots G(q_{s}^{\alpha_{s}})\nonumber \\
&=\frac{1}{m(\beta_{1})!\cdots m(\beta_{d})!}
\sum_{\begin{subarray}{c}q_{1},\ldots, q_{s}\leq z\\ distinct\\ ((\alpha_{1},\ldots, \alpha_{s}): \textit{a representative of $N_{l}(k,s)$)}\end{subarray}}
G(q_{1}^{\alpha_{1}})\cdots G(q_{s}^{\alpha_{s}}). \label{NKS-wa-2}
\end{align}
For each representative $(\alpha_{1},\ldots, \alpha_{s})\in N_{l}(k,s)$, we shall
put labels $i(l,l^{\prime})$ on indexes $i$ of $\alpha_{i}$ as
\begin{align*}
\beta_{1}&= \alpha_{i(1,1)}, \alpha_{i(1,2)}, \ldots, \alpha_{i(1,m(\beta_{1}))},\\
\beta_{2}&= \alpha_{i(2,1)}, \alpha_{i(2,2)}, \ldots, \alpha_{i(2,m(\beta_{2}))},\\
&\cdots\\
\beta_{d}&= \alpha_{i(d,1)}, \alpha_{i(d,2)}, \ldots, \alpha_{i(d,m(\beta_{d}))}.
\end{align*}
Let $S_{s}$ be the set of permutations of $\{1,2,\ldots, s\}$. For any $\sigma \in S_{s}$ and $\beta_{j}$ ($j=1,\ldots, d$) we shall set
\begin{align*}
L(\sigma, \beta_{j}):=\{ \nu \, |\, \nu=\sigma^{-1}(i(j,1)), \sigma^{-1}(i(j,2)),\ldots, \sigma^{-1}(i(j,m(\beta_{j})))\}.
\end{align*}
Here, we shall remark that for a $N_{l}(k,s)$ if $\sigma$ and $\sigma^{\prime}\in S_{s}$ satisfy
\begin{align}\label{Id-NKS}
L(\sigma, \beta_{1})=L(\sigma^{\prime},\beta^{\prime}), \ldots, L(\sigma, \beta_{d})=L(\sigma^{\prime}, \beta_{d}),
\end{align}
then
\begin{align*}
\sum_{q_{\sigma(1)}<\cdots <q_{\sigma(s)}\leq z}G(q_{1}^{\alpha_{1}})\cdots G(q_{s}^{\alpha_{s}})
=\sum_{q_{\sigma^{\prime}(1)}<\cdots <q_{\sigma^{\prime}(s)}\leq z}G(q_{1}^{\alpha_{1}})\cdots G(q_{s}^{\alpha_{s}}).
\end{align*}
We shall define an equivalence relation $\simeq_{2}$ on $S_{s}$ with respect to a class $N_{l}(k,s)$, for $\sigma$ and $\sigma^{\prime}\in S_{s}$,
if the $d$ many equalities (\ref{Id-NKS}) hold then we write $\sigma \simeq_{2} \sigma^{\prime}$. We have the partition of $S_{s}$ as
\begin{align*}
S_{s}=\bigsqcup_{j=1}^{\frac{s!}{m(\beta_{1})!\cdots m(\beta_{d})!}}T_{j}(\beta_{1},\ldots, \beta_{d}; N_{l}(k,s)).
\end{align*}
Note that $\sharp T_{j}(\beta_{1}, \ldots, \beta_{d}; N_{l}(k,s))=m(\beta_{1})!\cdots m(\beta_{d})!$.
Denote by $T(\beta_{1},\ldots, \beta_{d};N_{l}(k,s))$ the complete system of representatives of $S_{s}$ by the relation $\simeq_{2}$, with respect to a class $N_{l}(k,s)$. Then, we observe that
\begin{align*}
&\sum_{\begin{subarray}{c}q_{1}, \ldots, q_{s}\leq z\\ distinct \\ ((\alpha_{1},\ldots, \alpha_{s}): \textit{a representative of $N_{l}(k,s)$})\end{subarray}}
G(q_{1}^{\alpha_{1}})\cdots G(q_{s}^{\alpha_{s}})\\
&=\sum_{\sigma \in S_{s}}\sum_{q_{\sigma(1)}<\cdots <q_{\sigma(s)}\leq z}G(q_{1}^{\alpha_{1}})\cdots G(q_{s}^{\alpha_{s}})\\
&= m(\beta_{1})!\cdots m(\beta_{d})!\sum_{\tau\in T(\beta_{1}, \ldots, \beta_{d}; N_{l}(k,s))}
    \sum_{q_{\tau(1)}<\cdots <q_{\tau(s)}\leq z}G(q_{1}^{\alpha_{1}})\cdots G(q_{s}^{\alpha_{s}})\\
&=m(\beta_{1})!\cdots m(\beta_{d})!\sum_{\forall (\alpha_{1}, \ldots, \alpha_{s})\in N_{l}(k,s)}\sum_{q_{1}<\cdots <q_{s}\leq z}G(q_{1}^{\alpha_{1}})\cdots G(q_{s}^{\alpha_{s}}),
\end{align*}
which proves (\ref{NKS-wa-2}). Therefore, combining (\ref{NKS-wa-1}) and (\ref{NKS-wa-2}) we reach the assertion of Proposition \ref{K-prop-2}. 
\end{proof}

Finally, we shall derive Theorem \ref{K-theorem} from Propositions \ref{K-prop-1} and \ref{K-prop-2}.
\begin{proof}[Proof of Theorem \ref{K-theorem}]
Recall the assumption that $k\geq 2$ is a fixed integer and $z\leq x^{1/k}$. Then, by (\ref{23456-formula}) and (\ref{23456-error}) we have
\begin{align}
&\sum_{n\leq x}(\omega_{z}(n)-\log\log z)^{k}\nonumber \\
&=\sum_{l=2}^{k}\binom{k}{l}B^{k-l}\sum_{n\leq x}\left(\sum_{p\leq z}f_{p}(n)\right)^{l} +B^{k}x +O(\pi(z))\nonumber\\
&\quad +\begin{cases}
O\left(\frac{x(\log\log z)^{\frac{k-2}{2}}}{\log z}\right) & (k\geq 2, even)\\
O\left(\frac{x(\log\log z)^{\frac{k-1}{2}}}{\log z}\right) & (k\geq 3, odd)
\end{cases}.\label{aoaoao}
\end{align}
Next, we shall investigate $\sum_{n\leq x}\left(\sum_{n\leq x}f_{p}(n)\right)^{k}$. In (\ref{hoshi-6}),
using the estimate $\pi(z)\ll \frac{z}{\log z}$ we see that
\begin{align}\label{131313}
\sum_{n\leq x}\left(\sum_{p\leq z}f_{p}(n)\right)^{k}
=x\left(\sum_{\begin{subarray}{c}p_{1},\ldots, p_{k}\leq z\\ p_{1}\cdots p_{k}: \textit{square-full}\end{subarray}}G(p_{1}\cdots p_{k})\right)+O\left(\frac{x}{\log z}\right),
\end{align}
where the $O$-constant depends on $k$, but $k\geq 2$ is fixed. In (\ref{I-wa}), we shall apply Proosition \ref{K-prop-2} at fisrt, and Propositon \ref{K-prop-1} in the next,
then,  
\begin{align*}
&\sum_{\begin{subarray}{c}p_{1},\ldots, p_{k}\leq z\\ p_{1}\cdots p_{k}: \textit{square-full}\end{subarray}}G(p_{1}\cdots p_{k})\\
&=\sum_{1\leq s \leq k/2}\left(\sum_{l=1}^{n(k,s)}\frac{c(N_{l}(k,s))}{m(\beta_{1})!\cdots m(\beta_{d})!}
\hspace{-1cm}\sum_{\begin{subarray}{c}q_{1}, \ldots, q_{s}\leq z\\ distinct \\ ((\alpha_{1},\ldots, \alpha_{s}): \textit{a representative of $N_{l}(k,s)$})\end{subarray}}
\hspace{-1cm}G(q_{1}^{\alpha_{1}})\cdots G(q_{s}^{\alpha_{s}})\right)\\
&=\sum_{j=0}^{\left[\frac{k}{2}\right]}\mathcal{L}_{j}^{\prime}(\log\log z)^{j}+O\left(\frac{(\log\log z)^{\left[\frac{k}{2}\right]-1}}{\log z}\right).
\end{align*}
Combining the above and (\ref{131313}) we obtain
\begin{align*}
\sum_{n\leq x}\left(\sum_{p\leq z}f_{p}(n)\right)^{k}
=x\sum_{j=0}^{\left[\frac{k}{2}\right]}\tilde{a}_{j}(\log\log z)^{j}+O\left(\frac{x(\log\log z)^{\left[\frac{k}{2}\right]-1}}{\log z}\right),
\end{align*}
where $\tilde{a}_{j}$ are constants. Replacing $k$ with $l\, (l\leq k)$ in the above, and applying them to (\ref{aoaoao})
we obtain the form (\ref{K-theorem-12}). To obtain the expression (\ref{K-theorem-22}), we shall use Proposition \ref{K-prop-2}.
If $k\geq 2$ is even, then we have for $a_{\left[\frac{k}{2}\right]}=a_{\frac{k}{2}}$ in (\ref{K-theorem-12}) that
\begin{align*}
a_{\frac{k}{2}}=\sum_{j=1}^{n\left(k, \frac{k}{2}\right)}\frac{c\left(N_{j}\left(k,\frac{k}{2}\right)\right)}{m(\beta_{1})!\cdots m(\beta_{d})!}
=\frac{c\left(N\left(k,\frac{k}{2}\right)\right)}{\left(\frac{k}{2}\right)!}=\frac{k!}{\left(\frac{k}{2}\right)!2^{k/2}}=C_{k}.
\end{align*}
If $k\geq 3$ is odd (note that (\ref{aoaoao})), then in (\ref{K-theorem-12}) we have
\begin{align*}
&a_{\left[\frac{k}{2}\right]}=a_{\frac{k}{2}-\frac{1}{2}}\\
&=\sum_{j=1}^{n\left(k,\frac{k-1}{2}\right)}\frac{c\left(N_{j}\left(k,\frac{k-1}{2}\right)\right)}{m(\beta_{1})!\cdots m(\beta_{d})!}
  +kB\sum_{j=1}^{n\left(k-1,\frac{k-1}{2}\right)}\frac{c\left(N_{j}\left(k-1,\frac{k-1}{2}\right)\right)}{m(\beta_{1})!\cdots m(\beta_{d})!}\\
&=\frac{c\left(N\left(k,\frac{k-1}{2}\right)\right)}{1!\left(\frac{k-3}{2}\right)!}+kBC_{k-1}\\
&=\frac{k!}{3!2^{\frac{k-3}{2}}\left(\frac{k-3}{2}\right)!}+kBC_{k-1}\\
&=\tilde{C}_{k}+kBC_{k-1}.
\end{align*}
We see that the expression (\ref{K-theorem-22}) is consistent with leading coefficient of (a) and (b)-(I) in Theorem \ref{Hatsumi-1}, respectively.
\end{proof}

\section{A framework of Theorem \ref{Hatsumi-1}}
In this final section of this notes, we shall attempt to find a framework of Theorem \ref{Hatsumi-1}.
Our concern is what causes the appearence of the main term in Theorem \ref{Hatsumi-1}.

Let $\mathcal{P}$ be a set of distinct primes, and define $\omega_{\mathcal{P}}(n)$ by

\begin{align}\label{def-omega-P}
\omega_{\mathcal{P}}(n):=\sum_{\begin{subarray}{c}p|n\\ p\in\mathcal{P}\end{subarray}}1
\end{align}
which is a generalization of $\omega(n)=\sum_{p|n}1$. It is natural to ask a result corresponding to Theorem \ref{GS-th-1}  
(\cite[Theorem 1]{GS}) on $\omega_{\mathcal{P}}(n)$. 
The method of the proof of Theorem \ref{GS-th-1} is formulated, actual, in a sieve like setting, by Granville and Soundararajan, 
which is stated as Proposition 3 in \cite[p.~21]{GS}. 
In the proposition for odd integers $k\geq 3$, the `main term' does not appear. As (b) of Theorem \ref{Hatsumi-1}, 
it might be interesting to find the `main term' in Proposition 3 (in the case of odd integers $k$).

First, to state the formulation of Theorem \ref{GS-th-1} by Granville and Soundararajan we shall fix the notation and assumptions.

Let $x\geq 1$ be a real number and $\mathcal{A}(x)=\{a_{n}\}$ a sequence which may depend on $x$ of natural numbers those are not necessarily to distinct.
We may suppose that $\sharp \mathcal{A}(x)=\sum_{a_{n}\in\mathcal{A}(x)}1 \to \infty$, as $x\to \infty$.
Let $z\geq 1$ be another parameter, and $\mathcal{P}(z)$ a set of distinct primes defined the parameter. Also, we shall suppose $\sharp \mathcal{P}(z)=\sum_{p\in\mathcal{P}(z)}1 \to\infty$, as $z\to\infty$. Moreover, we shall assume some properties for $\mathcal{A}(x)$ and $\mathcal{P}(z)$.

For any square-free integer $d\geq 1$ and any real $x\geq 1$ we write
\begin{align*}
\mathcal{A}_{d}(x):=\{a_{n}\in\mathcal{A}(x)\ :\ d|a_{_{n}}\}.
\end{align*}
Suppose the properties (H1) and (H2), below, and we shall define two functions $\mu_{\mathcal{P}}(z)$ and $\sigma_{\mathcal{P}}(z)$
(these functions are introduced in \cite[p.~20]{GS}).

\begin{enumerate}
\item[\bf (H1)] {\it For any $x\geq 1$ and any square-free integer $d\geq 1$, $\sharp\mathcal{A}_{d}(x)=\sum_{a_{n}\in\mathcal{A}_{d}(x)}1$
is `approximated' as
\begin{align*}
\sharp \mathcal{A}_{d}(x)=\frac{h(d)}{d}N(x) +r_{d}(x),
\end{align*}
here $h(n)$ is a multiplicative function satisfying 
\begin{align*}
0\leq h(n) \leq n \quad (\text{for any $n\geq 1$}),
\end{align*}
and $N(x)$ is a positive monotonic increasing function.}
\end{enumerate}

On the basis of the function $h(n)$ which is recognized in the assumption (H1), we shall define $\mu_{\mathcal{P}}(z)$ and $\sigma_{\mathcal{P}}(z)$ for $z \geq 1$: 
\begin{align*}
& \mu_{\mathcal{P}}(z):=\sum_{p\in\mathcal{P}(z)}\frac{h(p)}{p},\\
& \sigma_{\mathcal{P}}(z):=\left(\sum_{p\in \mathcal{P}(z)}\frac{h(p)}{p}\left(1-\frac{h(p)}{p}\right)\right)^{1/2}.
\end{align*}
Note that
\begin{align*}
0 \leq \sigma_{\mathcal{P}}(z)^{2} \leq \mu_{\mathcal{P}}(z)\leq \sharp\mathcal{P}(z).
\end{align*}

\begin{enumerate}
\item[\bf (H2)] {\it $\sigma_{\mathcal{P}}(z) \to\infty$ as $z\to\infty$.}
\end{enumerate} 

Moreover, we assume the relation the natural number $k$ and $\sigma_{\mathcal{P}}(z)$ as follows.
\begin{enumerate}
\item[\bf (H3)] {\it For any sufficiently large $z\geq 1$, 
the inequality $k\leq \sigma_{\mathcal{P}}(z)^{2/3}$ is satisfied.
}
\end{enumerate}
If $k$ is fixed, then (H3) is derived from (H2). However, recall that $k$ is not fixed in Theorems \ref{GS-th-1} and \ref{Hatsumi-1}.
Actually, (H3) is required in Proposition 3 in \cite[p.~21]{GS}. 

Henceforth, on the hypothes (H1)--(H3) and some additional assumptions, 
we shall investigate $\omega_{\mathcal{P}}(a)$ of (\ref{def-omega-P}) for $a\in\mathcal{A}(x)$ and $\mathcal{P}=\mathcal{P}(z)$.
We may state Proposition 3 in \cite{GS} as follows.

\begin{proposition}[{\cite[p.~21, Proposition 3]{GS}}]\label{GS-Prop-3}
Let $k\geq 2$ be any integer, $x\geq 1$ and $z\geq 1$ be sufficiently large real number. 
Under the hypotheses (H1)--(H3), we have, uniformly in $k$, $z$, and $x$,
the following estimates:
\begin{enumerate}
\item[\rm (a)]If $k\geq 2$ is even, then
\begin{align*}
&\sum_{a\in\mathcal{A}(x)}(\omega_{\mathcal{P}}(a)-\mu_{\mathcal{P}}(z))^{k}\\
&=C_{k}N(x) \sigma_{\mathcal{P}}(z)^{k}\left(1+O\left(\frac{k^{3}}{\sigma_{\mathcal{P}}(z)^{2}}\right)\right)
+O\left(\mu_{\mathcal{P}}(z)^{k}\sum_{d\in\mathcal{D}_{k}(\mathcal{P})}|r_{d}(x)|\right).
\end{align*}
\item[\rm (b)]If $k\geq 3$ is odd, then
\begin{align*}
&\sum_{a\in\mathcal{A}(x)}(\omega_{\mathcal{P}}(a)-\mu_{\mathcal{P}}(z))^{k}\\
&=O\left(k^{3/2}C_{k}N(x)\sigma_{\mathcal{P}}(z)^{k-1}\right)
+O\left(\mu_{\mathcal{P}}(z)^{k}\sum_{d\in\mathcal{D}_{k}(\mathcal{P})}|r_{d}(x)|\right).
\end{align*}
\end{enumerate}
Here, $C_{k}$ is defined by (\ref{teigi-C-j}), and $\mathcal{D}_{k}(\mathcal{P})$ denotes the set of all integers, including 1, which are the product of at most $k$ primes in $\mathcal{P}(z)$. 
\end{proposition}

\begin{remark} It is clear that the hypotheses (H1) and (H3) are required in \cite[p.~21, Proposition 3]{GS}, but, the necessity of (H2) might be unclear in \cite{GS}.
Note that the hypothesis (H2) is used to find the lower bound in \cite[p.~22, line 13]{GS}.
\end{remark}

\begin{remark} In (a) of Proposition \ref{GS-Prop-3}, 
if $k \asymp \sigma_{\mathcal{P}}(z)^{2/3}$, the `main term' $C_{k}N(x)\sigma_{\mathcal{P}}(z)^{k}$ of the result (a) is absorbed by the $O$-term. 
\end{remark}

We shall attempt to take a `main term' out of the estimate (b) in Proposition \ref{GS-Prop-3}, as we have shown Theorem \ref{Hatsumi-1} (b). To this end, probably,
we have to ask more four hypotheses (H4)--(H7):

\begin{enumerate}
\item[\bf (H4)] {\it For any prime $p$ in $\mathcal{P}(z)$, the inequalities $0\leq \frac{h(p)}{p}\leq \frac{1}{2}$ are satisfied.}
\item[\bf (H5)] {\it For any sufficiently large $z\geq 1$,  there exists a prime $p^{*}\in\mathcal{P}(z)$ satisfying
                 \begin{align*}
                 \frac{h(p)}{p}\left(1-\frac{h(p)}{p}\right)\left(1-2\frac{h(p)}{p}\right)\quad ({\it for}\ p\geq p^{*}\ {\it and}\ p\in\mathcal{P}(z))
                 \end{align*}
                 is monotonic decreasing.

Furthermore, the set $\{p\in\mathcal{P}(z)\,|\, p\geq p^{*}\}$ has at least $\frac{k-1}{2}$ distinct primes.
}

\item[\bf (H6)] {\it Let $\pi_{l}(\mathcal{P})$ denote the $l$th smallest prime in $\mathcal{P}(z)$. For the prime $p^{*}$ in (H5), 
we write $p^{*}=\pi_{l^{*}}(\mathcal{P})$. As for the number $l^{*}$, the following properties are satisfied:
\begin{align*}
\pi_{l^{*}+\frac{k-3}{2}}(\mathcal{P})\in \mathcal{P}(z),\quad \text{and} \quad l^{*}=O(k).
\end{align*}}
\item[\bf (H7)] {\it Let
\begin{align*}
\tilde{\sigma}_{\mathcal{P}}(z):=\left(\sum_{p\in\mathcal{P}(z)}\left(\frac{h(p)}{p}\right)^{2}\left(1-\frac{h(p)}{p}\right)\right)^{1/2}.
\end{align*}
There exists an absolute constant $c_{1}>0$ satisfying $\tilde{\sigma}_{\mathcal{P}}(z)\leq c_{1}$.}
\end{enumerate}

\bigskip

Under the hypotheses (H1)--(H7), in the case of odd integer $k \geq 3$, we find a main term in (b) of Proposition \ref{GS-Prop-3}, as follows.
\begin{proposition}\label{GS-prop-3-odd}
Let $k\geq 3$ be an odd integer. 
For sufficiently large $x\geq 1$ and $z\geq 1$,
on the hypotheses (H1)--(H7), we have, uniformly in $k$, $z$, and $x$,
\begin{align*}
\sum_{a\in\mathcal{A}(x)}(\omega_{\mathcal{P}}(a)-\mu_{\mathcal{P}}(z))^{k}
&=\tilde{C}_{k}N(x) \sigma_{\mathcal{P}}(z)^{k-1}
+O\left((k^{3/2}C_{k})N(x)\sigma_{\mathcal{P}}(z)^{k-3}k^{3}\right)\\
&\quad +O\left(\mu_{\mathcal{P}}(z)^{k}\sum_{d\in\mathcal{D}_{k}(\mathcal{P})}|r_{d}(x)|\right).
\end{align*}
Here, $\tilde{C}_{k}$ is defined by (\ref{teigi-C-j-tilde}), and  $\mathcal{D}_{k}(\mathcal{P})$ denotes the set of all integers, including 1, 
which are the product of at most $k$ primes in $\mathcal{P}(z)$.
\end{proposition}

\begin{remark} If $k \asymp \sigma_{\mathcal{P}}(z)^{2/3}$, the `main term' $C_{k}N(x)\sigma_{\mathcal{P}}(z)^{k-1}$ is absorbed by the $O$-term. 
\end{remark}

To prove Proposition \ref{GS-prop-3-odd} we shall show some lemmas.

\begin{lemma}\label{lemma-I}
As for $\mathcal{A}(x)$, $\mathcal{P}(z)$, and the hypotheses (H1), we shall define $f_{p}(a)$ ($a\in\mathcal{A}(x)$, $p\in\mathcal{P}(z)$) by
\begin{align}\label{hori-mina-maki}
f_{p}(a):=\begin{cases}
1-\frac{h(p)}{p}& (p|a)\\
-\frac{h(p)}{p} & (p\nmid a)
\end{cases}.
\end{align}
Also, for any $a\in\mathcal{A}(x)$ and  $p_{1},\ldots, p_{l}\in\mathcal{P}(z)$, we write
\begin{align*}
f_{p_{1}\cdots p_{l}}(a):=f_{p_{1}}(a)\cdots f_{p_{l}}(a).
\end{align*} 

For $\omega_{\mathcal{P}}(a)$ ($a\in\mathcal{A}(x)$, $\mathcal{P}=\mathcal{P}(z)$) which is defined in (\ref{def-omega-P}), we have
\begin{align}\label{al-41}
\omega_{\mathcal{P}}(a)=\sum_{p\in\mathcal{P}(z)}f_{p}(a)+\mu_{\mathcal{P}}(z).
\end{align}
\end{lemma}
\begin{proof}
As we have seen in (\ref{ogai}), by (\ref{hori-mina-maki}) we obtain the assertion:
\begin{align*}
\omega_{\mathcal{P}}(a)&=\sum_{\begin{subarray}{c}p|a\\ p\in\mathcal{P}(z)\end{subarray}} \left(1-\frac{h(p)}{p}+\frac{h(p)}{p}\right)
=\sum_{\begin{subarray}{c}p|a\\ p\in\mathcal{P}(z)\end{subarray}} \left(1-\frac{h(p)}{p}\right) 
 +\sum_{\begin{subarray}{c}p|a\\ p\in\mathcal{P}(z)\end{subarray}} \frac{h(p)}{p}\\ 
&=\sum_{\begin{subarray}{c}p|a\\ p\in\mathcal{P}(z)\end{subarray}} \left(1-\frac{h(p)}{p}\right)
  +\sum_{p\in\mathcal{P}(z)}\frac{h(p)}{p} -\sum_{\begin{subarray}{c}p \nmid a\\ p\in\mathcal{P}(z)\end{subarray}} \frac{h(p)}{p}\\
&= \sum_{p\in\mathcal{P}(z)}f_{p}(a)+\mu_{\mathcal{P}}(z).
\end{align*}
\end{proof}

Our task is clarified by Lemma \ref{lemma-I}.
Let $k\geq 1$ be any integer. Under the hypothesis (H1), applying (\ref{al-41}) of Lemma \ref{lemma-I} we observe that
\begin{align}
\sum_{a\in\mathcal{A}(x)}(\omega_{\mathcal{P}}(a)-\mu_{\mathcal{P}}(z))^{k}
&=\sum_{a\in\mathcal{A}(x)}\left(\sum_{p\in\mathcal{P}(z)}f_{p}(a)\right)^{k}\nonumber \\
&=\sum_{p_{1},\ldots, p_{k}\in\mathcal{P}(z)}\sum_{a\in\mathcal{A}(x)}f_{p_{1}\cdots p_{k}}(a). \label{hori-maki}
\end{align}

On the right-hand side of (\ref{hori-maki}), for each $k$-tuple $\{p_{1},\ldots, p_{k}\}$ in $\sum_{p_{1},\ldots, p_{k}\in\mathcal{P}(z)}$ we write
$r=p_{1}\cdots p_{k}$.
Express the prime factorization of $r$ as
\begin{align*}
r=\prod_{i=1}^{s}q_{i}^{\alpha_{i}}\quad (\alpha_{i}\geq 1,\ q_{i}: \textit{distinct primes}),
\end{align*}
and put the kernel of $r$ as
\begin{align*}
R:=\prod_{i=1}^{s}q_{i}.
\end{align*}

We define $\mathcal{G}(r)$ and $\mathcal{E}(r)$ for each $\{p_{1},\ldots, p_{k}\}$ in $\sum_{p_{1},\ldots, p_{k}\in\mathcal{P}(z)}$, and $x\geq 1$ 
as follows.
\begin{align}
&\mathcal{G}(r):=\sum_{d|R}f_{p_{1}\cdots p_{k}}(d)\frac{h(d)}{d}\prod_{p|\frac{R}{d}}\left(1-\frac{h(p)}{p}\right),\label{4-def-Gr}\\
&\mathcal{E}(r):=\sum_{d|R}f_{p_{1}\cdots p_{k}}(d)\sum_{e|\frac{R}{d}}\mu(e)r_{de}(x). \label{4-def-Er}
\end{align}
Here, $h(\cdot)$ and $r_{de}(x)$ are admissible by (H1), and $\mu(\cdot)$ is the M{\"o}bius function. Note that
$\mathcal{G}(r)$ is independent of $x$, and $\mathcal{E}(r)$ depends on $x$.
Moreover, remark that for the greatest common divisor $\delta:=(a, R)$ in (\ref{hori-maki}), it holds
\begin{align*}
f_{r}(a)=f_{r}(\delta),
\end{align*}
which is stated in \cite[p.~21]{GS}.

Concerning (\ref{hori-maki}), we have the next lemma.  
\begin{lemma}\label{lemma-II}
Keep the notation  above.
Let $k\geq 1$ be any integer, also $x\geq 1$ and $z\geq 1$ real numbers. On the hypothesis (H1), we have
\begin{align}\label{hoshi}
\sum_{a\in\mathcal{A}(x)}\left(\omega_{\mathcal{P}}(a)-\mu_{\mathcal{P}}(z)\right)^{k}
=N(x) \sum_{p_{1},\ldots, p_{k}\in\mathcal{P}(z)}\mathcal{G}(r)+\sum_{p_{1},\ldots, p_{k}\in\mathcal{P}(z)}\mathcal{E}(r).
\end{align}
\end{lemma}
\begin{proof} By (\ref{hori-maki}) and (H1) we observe that
\begin{align*}
&\sum_{a\in\mathcal{A}(x)}\left(\omega_{\mathcal{P}}(a)-\mu_{\mathcal{P}}(z)\right)^{k}
=\sum_{p_{1},\ldots, p_{k}\in\mathcal{P}(z)}\sum_{d|R}\sum_{\begin{subarray}{c}a\in\mathcal{A}(x)\\ (a,R)=d\end{subarray}}f_{p_{1}\cdots p_{k}}(d)\\
&=\sum_{p_{1},\ldots, p_{k}\in\mathcal{P}(z)}\sum_{d|R}f_{p_{1}\cdots p_{k}}(d)\sum_{a\in\mathcal{A}(x)}\sum_{e|\left(\frac{a}{d}, \frac{R}{d}\right),\, d|a}\mu(e)\\
&=\sum_{p_{1}, \ldots, p_{k} \in\mathcal{P}(z)}\sum_{d|R}f_{p_{1}\cdots p_{k}}(d) \sum_{e|\frac{R}{d}}\mu(e) \sharp \mathcal{A}_{de}(x)\\
&=\sum_{p_{1},\ldots, p_{k} \in\mathcal{P}(z)}\sum_{d|R}f_{p_{1}\cdots p_{k}}(d)
  \left(\frac{h(d)}{d}N(x)\sum_{e|\frac{R}{d}}\frac{\mu(e)h(e)}{e} +\sum_{e|\frac{R}{d}}\mu(e)r_{de}(x)\right)\\
& \quad\quad  (\textit{by (H1)})\\
&= N(x)\sum_{p_{1},\ldots, p_{k}\in\mathcal{P}(z)}\mathcal{G}(r) +\sum_{p_{1},\ldots, p_{k}\in\mathcal{P}(z)}\mathcal{E}(r) \quad 
(\textit{by (\ref{4-def-Gr}) and (\ref{4-def-Er})}).
\end{align*}
\end{proof}

To examine the sum of $\mathcal{G}(r)$ in (\ref{hoshi}), we shall arrange properties on $\mathcal{G}(r)$. 
\begin{lemma}\label{lemma-III}
Under the hypothesis (H1), let $r$ and $r^{\prime}$ are products of primes in $\mathcal{P}(z)$.
Then, we have the following properties of $\mathcal{G}(r)$.
\begin{enumerate}
\item[\rm (i)] If $(r,r^{\prime})=1$, then $\mathcal{G}(rr^{\prime})=\mathcal{G}(r)\mathcal{G}(r^{\prime})$.
\item[\rm (ii)] $($\cite[p.~22, (14)]{GS}$)$
\begin{align*}
\mathcal{G}(r)=\prod_{q^{\alpha}||r}\left(\frac{h(q)}{q}\left(1-\frac{h(q)}{q}\right)^{\alpha}+\left(-\frac{h(q)}{q}\right)^{\alpha}\left(1-\frac{h(q)}{q}\right)\right).
\end{align*}
\item[\rm (iii)] $($\cite[p.~22, line 6]{GS}$)$ If $r=p_{1}\cdots p_{k}$ is not square-full, then $\mathcal{G}(r)=0$.
\item[\rm (iv)] Let $q$ be a prime in $\mathcal{P}(z)$. If $\alpha\geq 2$ is an even integer, then
\begin{align*}
0\leq \mathcal{G}(q^{\alpha}) \leq \frac{h(q)}{q}\left(1-\frac{h(q)}{q}\right).
\end{align*}
If $\alpha \geq 3$ is an odd integer, then
\begin{align*}
-\frac{h(q)}{q}\left(1-\frac{h(q)}{q}\right) \leq \mathcal{G}(q^{\alpha}) \leq \frac{h(q)}{q}\left(1-\frac{h(q)}{q}\right).
\end{align*}
For any integer $\alpha \geq 2$, it holds that
\begin{align*}
\left|\mathcal{G}(q^{\alpha})\right| \leq \frac{h(q)}{q}\left(1-\frac{h(q)}{q}\right).
\end{align*}
\item[\rm (v)] Assume also (H4). For any odd integer $\alpha\geq 3$, the inequalities hold
\begin{align*}
0\leq \mathcal{G}(q^{\alpha}) \leq \frac{h(q)}{q}\left(1-\frac{h(q)}{q}\right).
\end{align*}
\end{enumerate}
\end{lemma}
\begin{proof} 
It is not difficult to prove the assertions.
However, to make sure of the necessity of the hypothesis (H4) in our proof of Proposition \ref{GS-prop-3-odd}, we shall prove (iv) and (v), here.  

By (ii) and the inequalities $0\leq \frac{h(q)}{q}, 1-\frac{h(q)}{q}\leq 1$ (by (H1)) we see that for any even integer $\alpha \geq 2$
\begin{align*}
0\leq \mathcal{G}\left(q^{\alpha}\right)&=\frac{h(q)}{q}\left(1-\frac{h(q)}{q}\right)
                                                         \left(\left(1-\frac{h(q)}{q}\right)^{\alpha -1}+\left(\frac{h(q)}{q}\right)^{\alpha-1}\right)\\
&\leq \frac{h(q)}{q}\left(1-\frac{h(q)}{q}\right) \left( 1-\frac{h(q)}{q} +\frac{h(q)}{q}\right)\\
&=\frac{h(q)}{q}\left(1-\frac{h(q)}{q}\right).
\end{align*}
As for an odd integer $\alpha \geq 3$, we have by (H1) and (ii)
\begin{align*}
\mathcal{G}(q^{\alpha})=\frac{h(q)}{q}\left(1-\frac{h(q)}{q}\right)^{\alpha}
                           -\left(\frac{h(q)}{q}\right)^{\alpha} \left(1-\frac{h(q)}{q}\right) \leq \frac{h(q)}{q} \left(1-\frac{h(q)}{q}\right),
\end{align*}
and
\begin{align}
\mathcal{G}(q^{\alpha})&=\frac{h(q)}{q}\left(1-\frac{h(q)}{q}\right)
                          \left(\left(1-\frac{h(q)}{q}\right)^{\alpha -1}-\left(\frac{h(q)}{q}\right)^{\alpha-1}\right) \label{g-shita-26}\\
                       &\geq \frac{h(q)}{q}\left(1-\frac{h(q)}{q}\right)
                               \left(-\left(1-\frac{h(q)}{q}\right)^{\alpha -1}-\left(\frac{h(q)}{q}\right)^{\alpha-1}\right) \nonumber \\
                       &\geq \frac{h(q)}{q}\left(1-\frac{h(q)}{q}\right)
                               \left(-\left(1-\frac{h(q)}{q}\right)-\left(\frac{h(q)}{q}\right)\right) \nonumber\\
                       &= -\frac{h(q)}{q}\left(1-\frac{h(q)}{q}\right).
\end{align}
Therefore we obtain the assertion (iv) by (H1) only. Using the hypothesis (H4), we find that
$1-\frac{h(q)}{q}\geq \frac{1}{2}\geq \frac{h(q)}{q}$. Then, the right-hand side of (\ref{g-shita-26}) is non-negative, we reach the assertion (v),
on (H1) and (H4).
\end{proof}

\begin{remark} We  obtained (v) of Lemma \ref{lemma-III}, by (H1) and (H4).
Especially, we see that $\mathcal{G}\left(q^{3}\right)\geq 0$, which is used in the proof of Lemma \ref{lemma-IV}, below.
This is the reason behind the necessity of the hypothesis (H4). Recall Remark \ref{RRRemark} in Section 2.
\end{remark}

We shall examine the right-hand side of (\ref{hoshi}) in Lemma \ref{lemma-II}, except the case for $k=1$.
As we have seen in Remark \ref{remark-akademori} in Section 3, it is an exception. See Remark \ref{ma-reigai-828} below.
Let now $k\geq 2$ be any integer. We shall apply (iii) of Lemma \ref{lemma-III} to $\sum_{p_{1},\ldots, p_{k}\in\mathcal{P}(z)}\mathcal{G}(r)$ in Lemma \ref{lemma-II}.
Then,
\begin{align}
\sum_{p_{1},\ldots, p_{k}\in \mathcal{P}(z)} \mathcal{G}(r) 
&= \sum_{\begin{subarray}{c} p_{1},\ldots, p_{k}\in \mathcal{P}(z)\\ p_{1}\cdots p_{k}: \textit{square-full} \end{subarray}} \mathcal{G}(p_{1}\cdots p_{k}) \nonumber\\
&= \sum_{s\leq \frac{k}{2}}\sum_{\begin{subarray}{c}q_{1}<\cdots <q_{s}\\ q_{i}\in\mathcal{P}(z)\end{subarray}}
                         \sum_{\begin{subarray}{c}\alpha_{1}+\cdots +\alpha_{s}=k\\ \alpha_{i}\geq 2\end{subarray}}
                            \frac{k!}{\alpha_{1}!\cdots \alpha_{s}!}\mathcal{G}(q_{1}^{\alpha_{1}}\cdots q_{s}^{\alpha_{s}})\nonumber\\
&=:\begin{cases}
{\displaystyle \mathcal{I}_{\frac{k}{2}}(k) +\sum_{s\leq \frac{k}{2}-1}\mathcal{I}_{s}(k)} & (k\geq 2, even)\\
{\displaystyle \mathcal{I}_{\frac{k-1}{2}}(k) +\sum_{s\leq \frac{k-3}{2}}\mathcal{I}_{s}(k)} & (k\geq 3, odd)
\end{cases},\quad (say). \label{42-s}                 
\end{align}
First, we shall consider the case for odd integers $k\geq 3$.
\begin{lemma}\label{lemma-IV}
Let $k\geq 3$ be odd. 
Under the hypotheses (H1)--(H7), we have
\begin{align}\label{mugi-0}
\sum_{p_{1},\ldots, p_{k}\in\mathcal{P}(z)}\mathcal{G}(r)=\tilde{C}_{k}\sigma_{\mathcal{P}}(z)^{k-1}+O\left((k^{3/2}C_{k})\sigma_{\mathcal{P}}(z)^{k-3}k^{3}\right),
\end{align}    
uniformly in $k$ and $z$.
\end{lemma}
\begin{proof}
As for $\mathcal{I}_{\frac{k-1}{2}}(k)$ in (\ref{42-s}), we shall show
\begin{align}\label{mugi}
\mathcal{I}_{\frac{k-1}{2}}(k)=\tilde{C}_{k}\sigma_{\mathcal{P}}(z)^{k-1}\left(1+O\left(\frac{k^{2}}{\sigma_{\mathcal{P}}(z)^{2}}\right)\right),
\end{align}
under the assumptions. Since $\sharp \mathcal{P}(z)\geq k^{3}$ (by (H3)), $\mathcal{I}_{\frac{k-1}{2}}(k)$ is not an empty sum.
We observe that
\begin{align}
\mathcal{I}_{\frac{k-1}{2}}(k)
&=\sum_{\begin{subarray}{c}q_{1}<\cdots <q_{\frac{k-1}{2}}\\ q_{i}\in\mathcal{P}(z)\end{subarray}}
  \sum_{\begin{subarray}{c}\alpha_{1}+\cdots +\alpha_{\frac{k-1}{2}}=k\\ \alpha_{i}\geq 2\end{subarray}}
   \frac{k!}{\alpha_{1}!\cdots \alpha_{\frac{k-1}{2}}!}\mathcal{G}\left(q_{1}^{\alpha_{1}}\cdots q_{\frac{k-1}{2}}^{\alpha_{\frac{k-1}{2}}}\right)\nonumber\\
&=\frac{k!}{2^{\frac{k-1}{2}-1}\cdot 3!}\sum_{l=1}^{\frac{k-1}{2}}
   \left(
           \sum_{\begin{subarray}{c}q_{1}<\cdots <q_{\frac{k-1}{2}}\\q_{i}\in\mathcal{P}(z)\end{subarray}} \mathcal{G}\left(q_{l}^{3}\right)
             \prod_{\begin{subarray}{c}i=1\\ i\neq l\end{subarray}}^{\frac{k-1}{2}}\mathcal{G}\left(q_{i}^{2}\right)
    \right),\label{43-s}
\end{align}
in fact, there are $\frac{k-1}{2}$-tuples $\{\alpha_{1}, \ldots, \alpha_{\frac{k-1}{2}}\}$ satisfying that $\alpha_{1}+\cdots +\alpha_{\frac{k-1}{2}}=k$ and $\alpha_{i}\geq 2$:
\begin{align*}
\begin{cases}
 \alpha_{1}=3,\ \alpha_{2}=\cdots =\alpha_{\frac{k-1}{2}}=2, & \\
\quad \cdots & \\
\alpha_{1}=\cdots =\alpha_{\frac{k-1}{2}-1}=2,\ \alpha_{\frac{k-1}{2}}=3. &
\end{cases}
\end{align*}
On the right-hand side of (\ref{43-s}), note that
\begin{align*}
& \mathcal{G}\left(q_{l}^{3}\right)=\frac{h(q_{l})}{q_{l}}\left(1-\frac{h(q_{l})}{q_{l}}\right) \left(1-2 \frac{h(q_{l})}{q_{l}}\right) \geq 0,\\
& \mathcal{G}\left(q_{i}^{2}\right) =\frac{h(q_{i})}{q_{i}} \left(1-\frac{h(q_{i})}{q_{i}}\right) \geq 0,
\end{align*}
(by (H1), (H4), (ii) of Lemma \ref{lemma-III}) and $\mathcal{G}\left(q_{i}^{2}\right)\geq \mathcal{G}\left(q_{i}^{3}\right) \geq 0$, and
\begin{align*}
\tilde{C}_{k}=\frac{k!}{2^{\frac{k-1}{2}}3}\frac{k-1}{2}\frac{1}{\left(\frac{k-1}{2}\right)!}\quad (by\, (\ref{teigi-C-j-tilde})).
\end{align*}
We have
\begin{align}\label{44-s}
\mathcal{I}_{\frac{k-1}{2}}(k)
\begin{cases}
\leq \tilde{C}_{k} {\displaystyle \sum_{\begin{subarray}{c}q_{1}, \ldots, q_{\frac{k-1}{2}}\in\mathcal{P}(z)\\ distinct\end{subarray}}} 
                    {\displaystyle \prod_{i=1}^{\frac{k-1}{2}}}\mathcal{G}\left(q_{i}^{2}\right)\\
\geq \tilde{C}_{k} {\displaystyle \sum_{\begin{subarray}{c}q_{1}, \ldots, q_{\frac{k-1}{2}}\in\mathcal{P}(z)\\ distinct\end{subarray}}} 
                    {\displaystyle \prod_{i=1}^{\frac{k-1}{2}}}\mathcal{G}\left(q_{i}^{3}\right)
\end{cases}.
\end{align}
(Recall that $\mathcal{P}(z)$ contains distinct $k^{3}$-primes, at least, by (H3).)
By (iii) of Lemma \ref{lemma-III}, we bound $\mathcal{I}_{\frac{k-1}{2}}(k)$, as
\begin{align}\label{45-s}
\mathcal{I}_{\frac{k-1}{2}}(k)
\leq  \tilde{C}_{k} \left(\sum_{q\in\mathcal{P}(z)}\frac{h(q)}{q}\left(1-\frac{h(q)}{q}\right)\right)^{\frac{k-1}{2}}
=\tilde{C}_{k}\sigma_{\mathcal{P}}(z)^{k-1}.
\end{align}
We emphasize, here, 
that two hypothese (H1) and (H3) are only required to obtain the upper bound of $\mathcal{I}_{\frac{k-1}{2}}(k)$.
However, to obtain a lower bound of $\mathcal{I}_{\frac{k-1}{2}}(k)$, the seven hypotheses (H1)--(H7) are required.

First, we shall take a prime $p^{*}\in\mathcal{P}(z)$ from the hypothesis (H5). Then,
\begin{align}
\mathcal{I}_{\frac{k-1}{2}}(k)
& \geq \tilde{C}_{k}\sum_{\begin{subarray}{c}p^{*}\leq q_{1},\ldots , q_{\frac{k-1}{2}}\in\mathcal{P}(z) \\ q_{i}: distinct\end{subarray}}
                  \prod_{i=1}^{\frac{k-1}{2}} \mathcal{G} \left(q_{i}^{3}\right) \nonumber \\
&= \tilde{C}_{k}\sum_{\begin{subarray}{c}p^{*}\leq q_{1},\ldots , q_{\frac{k-1}{2}-1}\in\mathcal{P}(z) \\ q_{i}: distinct\end{subarray}}
                 \prod_{i=1}^{\frac{k-1}{2}-1}\mathcal{G}\left(q_{i}^{3}\right)
                 \sum_{\begin{subarray}{c}p^{*}\leq p \in\mathcal{P}(z) \\ p \neq q_{i} (i=1,\ldots, \frac{k-1}{2}-1)\end{subarray}}\mathcal{G}\left(p^{3}\right). \label{46-s}
\end{align}
Here, by the decreasing monotonicity of $\mathcal{G}(p^{3})$ (by (H5)), we find that
\begin{align*}
\sum_{\begin{subarray}{c}p^{*}\leq p \in\mathcal{P}(z) \\ p \neq q_{i} (i=1,\ldots, \frac{k-1}{2}-1)\end{subarray}}\mathcal{G}\left(p^{3}\right)
\geq \sum_{\pi_{l^{*}+\frac{k-3}{2}(\mathcal{P})}\leq t \in\mathcal{P}(z)}\mathcal{G}\left(t^{3}\right).
\end{align*}  
Note that by (H6) the right-hand side is not an empty sum.
Repeating this evaluation in (\ref{46-s}), we get
\begin{align}
&(\textit{RHS of }(\ref{46-s}))\nonumber \\
&\geq \tilde{C}_{k} \left(\sum_{\pi_{l^{*}+\frac{k-3}{2}}(\mathcal{P})\leq p \in\mathcal{P}(z)} \mathcal{G}\left(p^{3}\right)\right)^{\frac{k-1}{2}} 
                                                                                  \nonumber\\
&\geq \tilde{C}_{k} \left(\sum_{p\in\mathcal{P}(z)}\mathcal{G}\left(p^{3}\right)
                          -\sum_{\begin{subarray}{c}p\in\mathcal{P}(z)\\ p \leq \pi_{l^{*}+\frac{k-3}{2}}(\mathcal{P})\end{subarray}}
                          \mathcal{G}\left(p^{3}\right)\right)^{\frac{k-1}{2}}
 \left(\textit{note $\mathcal{G}\left(p^{3}\right) \leq\frac{\sqrt{3}}{18}$}\right) \nonumber\\
&\geq \tilde{C}_{k}\left(\sum_{p\in\mathcal{P}(z)}g\left(p^{3}\right) -\frac{\sqrt{3}}{18}c_{*}k\right)^{\frac{k-1}{2}} 
 \quad (\textit{$c_{*}>0$ is a constant by (H6)})\nonumber \\
&\geq \tilde{C}_{k}\left(\sigma_{\mathcal{P}}(z)^{2}-2c_{1}^{2}-\frac{\sqrt{3}}{18}c_{*}k\right)^{\frac{k-1}{2}} \quad (\textit{by (H7)})\nonumber \\
&=\tilde{C}_{k}\left(\sigma_{\mathcal{P}}(z)^{2} \left(1-\frac{2c_{1}^{2}+\frac{\sqrt{3}}{18}c_{*}k}{\sigma_{\mathcal{P}}(z)^{2}}\right)\right)^{\frac{k-1}{2}} 
\ (\textit{$\geq 0$ by (H3) and (H2)})\nonumber \\
&= \tilde{C}_{k} \sigma_{\mathcal{P}}(z)^{k-1}\left(1+O\left(\frac{k^{2}}{\sigma_{\mathcal{P}}(z)^{2}}\right)\right) \quad  (\textit{by (H3)}). \label{50-s}
\end{align}
Therefore, from (\ref{44-s}), (\ref{45-s}), and (\ref{50-s}) we obtain the assertion (\ref{mugi}).

Next, in (\ref{42-s}), we shall derive that
\begin{align}
\sum_{s \leq \frac{k-3}{2}}\mathcal{I}_{s}(k) 
&= 
\begin{cases}
0 & (k=3)\\
O\left((k^{3/2}C_{k})\sigma_{\mathcal{P}}(z)^{k-3} k^{3}\right) & (k\geq 5, \textit{odd})
\end{cases}\nonumber \\
&\ll \left(k^{3/2}C_{k})\sigma_{\mathcal{P}}(z)^{k-3} k^{3}\right) \quad (k\geq 3, \textit{odd}). \label{al-51}
\end{align}
It is obvious for $k=3$. By (i), (iv) of   Lemma \ref{lemma-III}  we have for odd integers $k\geq 5$
\begin{align*}
\sum_{s\leq\frac{k-3}{2}}\mathcal{I}_{s}(k) 
&\ll \sum_{s\leq \frac{k-3}{2}} \sum_{\begin{subarray}{c}q_{1}<\cdots <q_{s}\\ q_{i}\in\mathcal{P}(z)\end{subarray}}
     \sum_{\begin{subarray}{c}\alpha_{1}+\cdots +\alpha_{s}=k\\ \alpha_{i}\geq 2\end{subarray}} \frac{k!}{\alpha_{1}! \cdots \alpha_{s}!} \prod_{i=1}^{s}
       \mathcal{G}\left(q_{i}^{2}\right) \\
&\leq \sum_{s\leq \frac{k-3}{2}} \frac{k!}{s!}\left(\sum_{q\in\mathcal{P}(z)}\frac{h(q)}{q} \left(1-\frac{h(q)}{q}\right)\right)^{s}
      \sum_{\begin{subarray}{c}\alpha_{1}+\cdots \alpha_{s}=k\\ \alpha_{i}\geq 2\end{subarray}}\frac{1}{\alpha_{1}!\cdots \alpha_{s}!}\\
&\leq \sum_{s\leq \frac{k-3}{2}}\frac{k!}{s!2^{s}}\sigma_{\mathcal{P}}(z)^{2s} \binom{k-s}{s} \quad (\textit{see}\ (\ref{atsui-atsui})) \\
&= \sum_{s\leq \frac{k-3}{2}} C_{k}\frac{\Gamma \left(\frac{k}{2}+1\right)2^{\frac{k}{2}-s}}{s!} \frac{(k-s)!}{s!(k-2s)!} 
   \frac{\left(\sigma_{\mathcal{P}}(z)^{2}\right)^{\frac{k-3}{2}}}{\left(\sigma_{\mathcal{P}}(z)^{2}\right)^{\frac{k-3}{2}-s}}.
\end{align*}
Here, we put $j=\frac{k-3}{2}-s$. Then,
\begin{align*}
(\textit{RHS})
=C_{k}\sigma_{\mathcal{P}}(z)^{k-3}\sum_{j=0}^{\frac{k-5}{2}} \frac{\Gamma \left(\frac{k}{2}+1\right)2^{\frac{3}{2}+j}}{\left(\frac{k-3}{2}-j\right)!}
 \frac{\left(\frac{k+3}{2}+j\right)!}{\left(\frac{k-3}{2}-j\right)!(2j+3)!} \frac{1}{\sigma_{\mathcal{P}}(z)^{2j}}.
\end{align*}
Applying (\ref{dango-1}), (\ref{dango-2}), (\ref{dango-3}), and (H3) we see that
\begin{align*}
&\ll k^{9/2}C_{k}\sigma_{\mathcal{P}}(z)^{k-3}\sum_{j=0}^{\frac{k-5}{2}}\frac{1}{(2j+3)!}\left(\frac{k^{3}}{\sigma_{\mathcal{P}}(z)^{2}}\right)^{j}\\
& \leq k^{9/2}C_{k}\sigma_{\mathcal{P}}(z)^{k-3}\sum_{j=0}^{\frac{k-5}{2}}\frac{1}{(2j+3)!}\\ 
& \ll \left(k^{3/2}C_{k}\right)\sigma_{\mathcal{P}}(z)^{k-3} k^{3}.
\end{align*}
Hence, by (\ref{42-s}), (\ref{mugi}), and (\ref{al-51}) we, finally,   reach the assertion (\ref{mugi-0}) of Lemma \ref{lemma-IV}.
\end{proof}

\begin{remark}\label{remark-ue-1}
On the hypotheses (H1) and (H3) we obtain the upper bound
\begin{align}\label{al-53}
\sum_{p_{1},\ldots, p_{k}\in\mathcal{P}(z)}\mathcal{G}(r) \ll \tilde{C}_{k}\sigma_{\mathcal{P}}(z)^{k-1},
\end{align}
by (\ref{42-s}), (\ref{45-s}), and (\ref{al-51}).
\end{remark}

We shall bound $\sum_{p_{1},\ldots, p_{k}\in\mathcal{P}(z)}\mathcal{E}(r)$ in Lemma \ref{lemma-II}.
\begin{lemma}[{\cite[p.~22, 23]{GS}}]\label{lemma-V}
Let $k\geq 2$ be any integer.
Under the hypotheses (H1) and (H3), 
as for (\ref{hoshi}) of Lemma \ref{lemma-II} we have
\begin{align}\label{GS-22-nazo}
\sum_{p_{1},\ldots, p_{k}\in\mathcal{P}(z)}\mathcal{E}(r) =O\left(\mu_{\mathcal{P}}(z)^{k}\sum_{d\in\mathcal{D}_{k}(\mathcal{P})}|r_{d}(x)|\right).
\end{align}
\end{lemma}
\begin{proof}
By the definitions (\ref{4-def-Er}) and (\ref{hori-mina-maki}) of $\mathcal{E}(r)$ and $f_{p_{1}\cdots p_{k}}(\cdot)$ we observe that
\begin{align}
|\mathcal{E}(r)| &\leq \sum_{d|R}|f_{p_{1}\cdots p_{k}}(d)|\sum_{e\left|\frac{R}{d}\right.}|r_{de}(x)| 
                 = \sum_{\delta |R}|r_{\delta}(x)|\sum_{\delta^{\prime}|\frac{R}{\delta}}|f_{p_{1}\cdots p_{k}}(\delta^{\prime})| \nonumber\\ 
                 &\leq \sum_{\delta |R}|r_{\delta}(x)|\sum_{\delta^{\prime}|\frac{R}{\delta}}\prod_{\begin{subarray}{c} p_{i}\\ p_{i}\nmid \delta^{\prime}\end{subarray}}\frac{h(p_{i})}{p_{i}}. \label{mai-zenigata}
\end{align}
Since $R=q_{1}\cdots q_{s}$ ($q_{1}<\cdots< q_{s}$, $s=\omega(R)$), we write $\delta=q_{i_{1}}\cdots q_{i_{i_{t}}}$, $t=\omega(\delta)$, also we put $j=\omega(\delta^{\prime})$.
Denote by $\tilde{j}$ the number of $p_{i}$ in $r=p_{1}\cdots p_{k}$ such that $p_{i} \nmid \delta^{\prime}$. 
We see that $0\leq s-t\leq \omega(R/\delta)\leq k$ and $\tilde{j}\leq k-j$. 
We observe that
\begin{align*}
(\textit{RHS of (\ref{mai-zenigata})}) &\leq \sum_{\delta |R} |r_{\delta}(x)| \sum_{\delta^{\prime}|\frac{R}{\delta}}\left(\sum_{p\in\mathcal{P}(z)}\frac{h(p)}{p}\right)^{\tilde{j}}
\nonumber\\
&\leq   \sum_{\delta |R} |r_{\delta}(x)| \sum_{\delta^{\prime}|\frac{R}{\delta}}\left(\mu_{\mathcal{P}}(z)\right)^{k-j} \quad (\textit{by (H3)}, \mu_{\mathcal{P}}(z)\geq 1)\nonumber \\
&= \sum_{\delta |R} |r_{\delta}(x)| \sum_{j=0}^{s-t} \binom{s-t}{j}\mu_{\mathcal{P}}(z)^{k-j} \nonumber\\
&= \sum_{\delta |R} |r_{\delta}(x)| \sum_{j=0}^{s-t} \binom{s-t}{j}\mu_{\mathcal{P}}(z)^{(s-t)-j}\mu_{\mathcal{P}}(z)^{k-(s-t)} \nonumber\\
&= \sum_{\delta |R} |r_{\delta}(x)|(1+\mu_{\mathcal{P}}(z))^{s-t}\mu_{\mathcal{P}}(z)^{k-(s-t)} \nonumber \\
&\leq (1+\mu_{\mathcal{P}}(z))^{k} \sum_{\delta |R} |r_{\delta}(x)| \nonumber \\
&\leq 2\mu_{\mathcal{P}}(z)^{k} \sum_{\delta |R}|r_{\delta}(x)| \quad (\textit{by (H3)}), 
\end{align*}
and we obtain the assertion (\ref{GS-22-nazo}).
\end{proof}

We can now obtain  Proposition \ref{GS-prop-3-odd} as follows.

\begin{proof}[Proof of Proposition \ref{GS-prop-3-odd}]
By (\ref{hoshi}) of Lemma \ref{lemma-II}, (\ref{mugi-0}) of Lemma \ref{lemma-IV}, (\ref{GS-22-nazo}) of Lemma \ref{lemma-V},
we obtain the assertion of Proposition \ref{GS-prop-3-odd}.
\end{proof}

\bigskip

Concerning the proof of Proposition \ref{GS-prop-3-odd}, we shall mention two remarks.
\begin{remark}\label{tokeigaugoku} 
It is enough to assume (H1) and (H3) to get that
\begin{align*}
\sum_{a\in\mathcal{A}(x)}(\omega_{\mathcal{P}}(a)-\mu_{\mathcal{P}}(z))^{k}
\ll \tilde{C}_{k}x\sigma_{\mathcal{P}}(z)^{k-1} +\mu_{\mathcal{P}}(z)^{k}\sum_{d\in\mathcal{D}_{k}(\mathcal{P})}|r_{d}(x)|,
\end{align*}
by use of (\ref{hoshi}), (\ref{al-53}), and (\ref{GS-22-nazo}). See, also Remark \ref{royalmilktea3bai}. 
\end{remark}

\begin{remark}\label{ma-reigai-828}
Proposition \ref{GS-prop-3-odd} does not include the case $k=1$. However, for $k=1$ we have easily that
\begin{align}\label{k=1-prop-4-2}
\sum_{a\in\mathcal{A}(x)}(\omega_{\mathcal{P}}(a)-\mu_{\mathcal{P}}(z))
=O\left(|r_{1}(x)|\mu_{\mathcal{P}}(z)\right)+O\left(\sum_{p\in\mathcal{P}(z)}|r_{p}(x)|\right).
\end{align}
In fact, by (\ref{hori-maki}) (or (\ref{al-41})), (\ref{hori-mina-maki}) and (H1) we have
\begin{align*}
&\sum_{a\in\mathcal{A}(x)}(\omega_{\mathcal{P}}(a)-\mu_{\mathcal{P}}(z))
=\sum_{p\in\mathcal{P}(z)}\sum_{a\in\mathcal{A}(x)}f_{p}(a)\\
&=\sum_{p\in\mathcal{P}(z)}
   \left(\sum_{\begin{subarray}{c}a\in\mathcal{A}(x)\\ p|a\end{subarray}}\left(1-\frac{h(p)}{p}\right) 
         +\sum_{a\in\mathcal{A}(x)}\left(-\frac{h(p)}{p}\right) 
         +\sum_{\begin{subarray}{c}a\in\mathcal{A}(x)\\ p|a \end{subarray}}\frac{h(p)}{p}\right)\\
&= \sum_{p\in\mathcal{P}(z)}\left(\sum_{\begin{subarray}{c}a\in\mathcal{A}(x)\\ p|a \end{subarray}}1-\frac{h(p)}{p}\sum_{a\in\mathcal{A}(x)}1\right)\\
&=\sum_{p\in\mathcal{P}(z)}\left(\frac{h(p)}{p}N(x)+r_{p}(x)-\frac{h(p)}{p}(N(x)+r_{1}(x))\right)\quad (\textit{by (H1)})\\
&=\sum_{p\in\mathcal{P}(z)}r_{p}(x) -r_{1}(x)\sum_{p\in\mathcal{P}(z)}\frac{h(p)}{p},
\end{align*}
which implies (\ref{k=1-prop-4-2}). The bound (\ref{k=1-prop-4-2}) is corresponding to (\ref{k=1-lemma-natsuki-1}) in Remark \ref{k=1-lemma-natsuki-1-remark}.
\end{remark}

We may consider a corresponding result of Proposition \ref{GS-prop-3-odd} for even integers $k\geq 2$.
In this case, probably, the hypotheses (H4) and (H7) are not required, and instead of (H5) and (H6)
the following (H5$^{\prime}$) and (H6$^{\prime}$) are asked:

\begin{enumerate}
\item[\bf (H5$^{\prime}$)] {\it For sufficiently large $z\geq 1$, there exists a prime $p^{*}\in\mathcal{P}(z)$ satisfying
\begin{align*}
\frac{h(p)}{p}\left(1-\frac{h(p)}{p}\right)\quad ({\it for}\ p\geq p^{*},\ p\in\mathcal{P}(z))
\end{align*} 
is monotonic decreasing.

Furthermore, the set $\{p\in\mathcal{P}\, |\, p\geq p^{*}\}$ contains at least $\frac{k}{2}$ distinct primes.
}
\item[\bf (H6$^{\prime}$)] {\it As in (H6), let $\pi_{l}(\mathcal{P})$ denotes the $l$th smallest prime in $\mathcal{P}(z)$, and  for $p^{*}$ in (H5$^{\prime}$), 
we write $p^{*}=\pi_{l^{*}}(\mathcal{P})$. As for $l^{*}$, the following properties are satisfied
\begin{align*}
\pi_{l^{*}+\frac{k-2}{2}}(\mathcal{P}) \in \mathcal{P} \quad \text{and}\quad l^{*}=O(k).
\end{align*}}
\end{enumerate}

Under the hypotheses (H1)--(H3), (H5$^{\prime}$) and (H6$^{\prime}$) we have the following proposition for even integers $k\geq 2$.
\begin{proposition}[{cf. Proposition.~\ref{GS-prop-3-odd}}]\label{GS-prop-3-EVEN}
Let $k\geq 2$ be an even integer. 
For sufficiently large $x\geq 1$ and $z\geq 1$, Under the hypotheses (H1)--(H3), (H5$^{\prime}$), and (H6$^{\prime}$), we have
uniformly in $k$, $z$, and $x$,
\begin{align}\label{al-70}
&\sum_{a\in\mathcal{A}(x)}(\omega_{\mathcal{P}}(a)-\mu_{\mathcal{P}}(z))^{k} \nonumber \\
&=C_{k}N(x)\sigma_{\mathcal{P}}(z)^{k}\left(1+O\left(\frac{k^{3}}{\sigma_{\mathcal{P}}(z)^{2}}\right)\right)
+O\left(\mu_{\mathcal{P}}(z)^{k}\sum_{d\in\mathcal{D}_{k}(\mathcal{P})}|r_{d}(x)|\right),
\end{align}
here, $C_{k}$ is the constant defined in (\ref{teigi-C-j}), and $D_{k}(\mathcal{P})$ is the set of all integers, including 1, which are the product of
at most $k$ primes in $\mathcal{P}(z)$.
\end{proposition}

Similar to the previous proof of Lemma \ref{lemma-IV}, we may obtain the next lemma.

\begin{lemma}[{cf.~Lemma \ref{lemma-IV}}]\label{lemma-IV-prime}
Let $k\geq 2$ be an even integer. For sufficiently large $z\geq 1$, on the hypotheses (H1)--(H3), (H5$^{\prime}$), and (H6$^{\prime}$)
we have uniformly in $k$ and $z$
\begin{align}\label{-70-lemma-IV-prime}
\sum_{p_{1},\ldots, p_{k}\in\mathcal{P}(z)}\mathcal{G}(r) =C_{k}\sigma_{\mathcal{P}}(z)^{k}
+O\left(C_{k}\sigma_{\mathcal{P}}(z)^{k-2}k^{3}\right).
\end{align}

\end{lemma}
\begin{proof}
Since we may prove this as Lemma \ref{lemma-IV}, we shall describe its proof, briefly. 
First, in (\ref{42-s}) we see that
\begin{align}\label{al-59}
\mathcal{I}_{\frac{k}{2}}(k)
=C_{k}\sum_{\begin{subarray}{c}q_{1},\ldots, q_{\frac{k}{2}} \in\mathcal{P}(z)\\ \textit{distinct} \end{subarray}}
 \prod_{j=1}^{\frac{k}{2}}\frac{h(q_{j})}{q_{j}}\left(1-\frac{h(q_{j})}{q_{j}}\right),
\end{align}
by Lemma \ref{lemma-III}. From this, we have at once
\begin{align}\label{al-60}
\mathcal{I}_{\frac{k}{2}}(k) \leq C_{k} \left(\sum_{p\in\mathcal{P}(z)}\frac{h(p)}{p}\left(1-\frac{h(p)}{p}\right)\right)^{\frac{k}{2}}
=C_{k}\sigma_{\mathcal{P}}(z)^{k}.
\end{align} 
On the other hand, (H5$^{\prime}$), (H6$^{\prime}$), (H3), and (H2), we observe that
\begin{align}
\mathcal{I}_{\frac{k}{2}}(k) 
&\geq C_{k} \left(\sum_{\pi_{l^{*}+\frac{k-2}{2}}(\mathcal{P})\leq p \in\mathcal{P}(z)}\mathcal{G}\left(p^{2}\right)\right)^{\frac{k}{2}} \nonumber \\
&\geq C_{k} \left(\sum_{p\in\mathcal{P}(z)}\mathcal{G}\left(p^{2}\right) 
                  -\sum_{\begin{subarray}{c}p\in\mathcal{P}(z)\\ p \leq \pi_{l^{*}+\frac{k-2}{2}} (\mathcal{P})\end{subarray}}\mathcal{G}\left(p^{2}\right)
                  \right)^{\frac{k}{2}}\nonumber\\
& \quad (\textit{note that $0\leq \mathcal{G}(p^{2})\leq 1/4$})\nonumber \\
&\geq C_{k}\sigma_{\mathcal{P}}(z)^{k} \left(1-\frac{\frac{1}{4}c_{*}k}{\sigma_{\mathcal{P}}(z)^{2}}\right)^{\frac{k}{2}} \quad (\textit{$>0$ by (H3) and H(2)})\nonumber \\
&= C_{k} \sigma_{\mathcal{P}}(z)^{k} \left(1+O\left(\frac{k^{2}}{\sigma_{\mathcal{P}}(z)^{2}}\right)\right) \quad (\textit{by (H3)}). \label{al-64}
\end{align} 
Hence, by (\ref{al-60}) and (\ref{al-64}) we get
\begin{align}\label{al-65}
\mathcal{I}_{\frac{k}{2}}(k)=C_{k} \sigma_{\mathcal{P}}(z)^{k} +O\left(C_{k}\sigma_{\mathcal{P}}(z)^{k-2} k^{2}\right).
\end{align}
Next, in (\ref{42-s}), for $k\geq 4$ we observe that
\begin{align}
\sum_{s\leq \frac{k}{2}-1}\mathcal{I}_{s}(k) 
&\ll \sum_{s\leq\frac{k-2}{2}} \frac{k!}{s!2^{s}} \sigma_{\mathcal{P}}(z)^{2s} \binom{k-s}{s}\nonumber \\
&= C_{k}\sum_{s\leq \frac{k-2}{2}}\frac{\Gamma \left(\frac{k}{2}+1\right)2^{\frac{k}{2}-s}}{s!}\frac{(k-s)!}{s!(k-2s)!}
   \frac{\left(\sigma_{\mathcal{P}}(z)^{2}\right)^{\frac{k-2}{2}}}{\left(\sigma_{\mathcal{P}}(z)^{2}\right)^{\frac{k-2}{2}-s}}\nonumber \\
&=C_{k}\sigma_{\mathcal{P}}(z)^{k-2}\sum_{j=0}^{\frac{k-4}{2}} 
                \frac{\left(\frac{k}{2}\right)! 2^{j+1}}{\left(\frac{k-2}{2}-j\right)!}
                \frac{\left(\frac{k}{2}+1+j\right)!}{\left(\frac{k-2}{2}-j\right)!(2j+2)!}
                \frac{1}{\left(\sigma_{\mathcal{P}}(z)^{2}\right)^{j}}. \label{Gaan}
\end{align}
Noting
\begin{align*}
\frac{\left(\frac{k}{2}\right)!2^{j+1}}{\left(\frac{k-2}{2}-j\right)!}\ll k^{j+1},
\quad  \frac{\left(\frac{k}{2}+1+j\right)!}{\left(\frac{k-2}{2}-j\right)!}\ll k\cdot k^{2j+1}, 
\end{align*}
and (H3) we have
\begin{align}
(\textit{RHS of } (\ref{Gaan}))
&\ll C_{k}\sigma_{\mathcal{P}}(z)^{k-2} k^{3}\sum_{j=0}^{\frac{k-4}{2}}\frac{1}{(2j+2)!}\left(\frac{k^{3}}{\sigma_{\mathcal{P}}(z)^{2}}\right)^{j} \nonumber \\
&\ll C_{k}\sigma_{\mathcal{P}}(z)^{k-2}k^{3}. \label{al-66}    
\end{align}
Obviously it is an empty sum for $k=2$.
Therefore, from (\ref{42-s}), (\ref{al-65}), and (\ref{al-66}) we reach the assertion (\ref{-70-lemma-IV-prime}) of Lemma \ref{lemma-IV-prime}.
\end{proof}

\begin{remark}
As we have mentioned in Remark \ref{remark-ue-1}, it is enough to ask the hypotheses (H1) and (H3) to obtain
the upper bound:
\begin{align}\label{al-69}
\sum_{p_{1},\ldots, p_{k}\in\mathcal{P}(z)}\mathcal{G}(r) \ll C_{k}\sigma_{\mathcal{P}}(z)^{k},
\end{align} 
by (\ref{42-s}), (\ref{al-60}), and (\ref{al-66}).
\end{remark}

We now prove Proposition \ref{GS-prop-3-EVEN}.
\begin{proof}[Proof of Proposition \ref{GS-prop-3-EVEN}]
Let $k\geq 2$ be an even integer. Note that under the hypothesis (H1) we have the identity (\ref{hoshi}).
Beside (H3), using Lemma \ref{lemma-V} in (\ref{hoshi}) we observe that
\begin{align}\label{al-58}
&\sum_{a\in\mathcal{A}(x)}(\omega_{\mathcal{P}}(a)-\mu_{\mathcal{P}}(z))^{k} \nonumber\\
&=N(x)\sum_{p_{1},\cdots, p_{k}\in\mathcal{P}(z)}\mathcal{G}(r) +O\left(\mu_{\mathcal{P}}(z)^{k}\sum_{d\in\mathcal{D}_{k}(\mathcal{P})}|r_{d}(x)|\right).
\end{align}
Under the hypotheses (H1)--(H3), (H5$^{\prime}$), and (H6$^{\prime}$) we shall apply (\ref{-70-lemma-IV-prime}) in Lemma \ref{lemma-IV-prime} to the above (\ref{al-58}).
Immediately we obtain the assertion (\ref{al-70}) of Proposition \ref{GS-prop-3-EVEN}.
\end{proof}

\begin{remark}\label{royalmilktea3bai}
Under the hypotheses (H1) and (H3),  we obtain the upper bound:
\begin{align*}
\sum_{a\in\mathcal{A}(x)}(\omega_{\mathcal{P}}(a)-\mu_{\mathcal{P}}(z))^{k}
\ll C_{k}N(x) \sigma_{\mathcal{P}}(z)^{k}+\mu_{\mathcal{P}}(z)^{k}\sum_{d\in\mathcal{D}_{k}(\mathcal{P})}|r_{d}(x)|
\end{align*}
by (\ref{al-58}), and (\ref{al-69}). See, also Remark \ref{tokeigaugoku}.
\end{remark}

Observing the proofs of Propositions \ref{GS-prop-3-odd} and \ref{GS-prop-3-EVEN}, we feel that
the results are too general. Actually, for each concrete problem 
we have to consider how to handle error terms in $\mu_{\mathcal{P}}(z)$, $\sigma_{\mathcal{P}}(z)$ and we have to bound $\sum_{d\in\mathcal{D}_{k}(\mathcal{P})}|r_{d}(x)|$.
However, probably, these two propositions indicate a direction in the study.

\bigskip


%

\bigskip

\noindent Tokuhon Makoto Minamide\\
Graduate School of Sciences and Technology for Innovation\\
Yamaguchi University\\
Yoshida 1677-1, Yamaguchi 753-8512, Japan\\
E-mail: minamide@yamaguchi-u.ac.jp\\

\noindent Haruka Sakai\\
Graduate School of Sciences and Technology for Innovation\\
Yamaguchi University\\
Yoshida 1677-1, Yamaguchi 753-8512, Japan\\
E-mail: e003vbv@yamaguchi-u.ac.jp\\

\noindent Yoshio Tanigawa\\
Nishizato 2-13-1, Meito, Nagoya 465-0084, Japan\\
E-mail: tanigawa@math.nagoya-u.ac.jp
\end{document}